\documentclass{amsart} 

\usepackage{amsmath,amssymb,amsthm} 
\usepackage{graphicx} 
\usepackage{subfig}
\usepackage[arrow, matrix, curve]{xy} 
\usepackage{xcolor} 
\usepackage{here}
\usepackage{pifont}
\usepackage{makecell} 
\usepackage{multirow} 
\usepackage{enumerate}
\usepackage{booktabs} 
\usepackage{siunitx} 
\usepackage{enumitem} 
\usepackage{hyperref} 
\usepackage{ulem}

\newcommand{\Z}{\mathbb{Z}}

\newcommand{\N}{\mathbb{N}}

\makeatletter
\newcommand{\Vast}{\bBigg@{2.5}} 
\makeatother

\makeatletter
\newtheorem*{rep@theorem}{\rep@title}
\newcommand{\newreptheorem}[2]{%
\newenvironment{rep#1}[1]{%
 \def\rep@title{#2 \ref{##1}}%
 \begin{rep@theorem}}%
 {\end{rep@theorem}}}
\makeatother

\newtheoremstyle{thm}{}{}{\itshape}{}{\bfseries}{}{ }{} 
\newtheoremstyle{definition}{}{}{}{}{\bfseries}{}{ }{} 

\theoremstyle{thm}
\newtheorem{Theorem}{Theorem}[section]
\newtheorem{thm}[Theorem]{Theorem}
\newreptheorem{thm}{Theorem}  
\newtheorem{lem}[Theorem]{Lemma}
\newreptheorem{lem}{Lemma}  
\newtheorem{prop}[Theorem]{Proposition}
\newtheorem{cor}[Theorem]{Corollary}
\newtheorem*{Theorem-ohne}{Theorem}

\theoremstyle{definition}

\newtheorem{rem}[Theorem]{Remark}
\newtheorem{ex}[Theorem]{Example}

\begingroup\expandafter\expandafter\expandafter\endgroup
\expandafter\ifx\csname pdfsuppresswarningpagegroup\endcsname\relax
\else
\fi

\definecolor{amaranth}{rgb}{0.9, 0.17, 0.31} 
\definecolor{carrotorange}{rgb}{0.93, 0.57, 0.13} 
\definecolor{citrine}{rgb}{0.89, 0.82, 0.04} 
\definecolor{dartmouthgreen}{rgb}{0.05, 0.5, 0.06} 
\definecolor{ballblue}{rgb}{0.13, 0.67, 0.8} 
\definecolor{ceruleanblue}{rgb}{0.16, 0.32, 0.75} 
\definecolor{amethyst}{rgb}{0.6, 0.4, 0.8} 
\definecolor{amber}{rgb}{1.0, 0.75, 0.0} 
\definecolor{burlywood}{rgb}{0.87, 0.72, 0.53} 

\numberwithin{equation}{section}

\begin{document}

\title{Complexity of equal 0-surgeries} 

\author{Tetsuya Abe}
\address{Otani University, Koyama-Kamifusacho, Kita-ku, Kyoto 603-8143, Japan. }
\email{tetsuyaabe2010@gmail.com }

\author{Marc Kegel}
\address{Universidad de Sevilla, Dpto.\ de Álgebra,
Avda.\ Reina Mercedes s/n,
41012 Sevilla, Spain}
\email{mkegel@us.es, kegelmarc87@gmail.com}

\author{Nicolas Weiss}
\address{Max Planck Institute for Mathematics in the Sciences, Inselstra{\ss}e 22, 04103 Leipzig, Germany.}
\email{nicolas.weiss@mis.mpg.de, nicolas.alexander.weiss@gmail.com}


\date{\today} 

\begin{abstract}
We say that two knots are friends if they share the same $0$-surgery. Two friends with different sliceness status would provide a counterexample to the $4$-dimensional smooth Poincar\'e conjecture. Here we create a census of all friends with small crossing numbers $c$ and tetrahedral complexities $t$, and compute their smooth $4$-genera. In particular, we compute the minimum of $c(K)+c(K')$ and of $t(K)+t(K')$ among all friends $K$ and $K'$. Along the way, we classify all $0$-surgeries of prime knots of at most $15$ crossings. 

Moreover, we determine for many friends in our census if their traces are equivalent or not. For that, we develop a new obstruction for two traces being homeomorphic coming from symmetry-exceptional slopes of hyperbolic knots. This is enough to also determine the minimum value of $c(K)+c(K')$ among all friends $K$ and $K'$ whose traces are not homeomorphic.
\end{abstract}
\keywords{$0$-surgeries, characterizing slopes, exceptional surgeries, knot traces} 

\makeatletter
\@namedef{subjclassname@2020}{%
  \textup{2020} Mathematics Subject Classification}
\makeatother

\subjclass[2020]{57K10; 57R65, 57R58, 57K14, 57K32} 


\maketitle


\section{Introduction}

From a given knot $K$ in the $3$-sphere $S^3$ we can construct new $3$- and $4$-manifolds by performing Dehn surgery on $K$ or attaching $2$-handles along $K$. In this article, we are interested in pairs of knots $K$ and $K'$ that share the same $0$-surgery. To construct the $0$-\textit{surgery} $K(0)$ of a knot $K$ we take the knot exterior $S^3\setminus \nu K$ of $K$ and glue a solid torus $V$ to it via a diffeomorphism such that the meridian of $V$ gets identified with the longitude of $K$ that is null-homologous in the knot exterior. We say that two (non-isotopic) knots $K$ and $K'$ are \textit{friends} if their $0$-surgeries are orientation-preserving diffeomorphic. 

We can also create a $4$-manifold from $K$, its \textit{trace} $X(K)$, obtained by attaching a $4$-dimensional $2$-handle to $D^4$ attached along $K$ with $0$-framing. By construction the boundary of $X(K)$ is diffeomorphic to the $0$-surgery $K(0)$ of $K$. In fact, it is known that the sliceness status of $K$ is determined by its trace, i.e.\ if two knots $K$ and $K'$ share the same trace, then they are either both slice or both non-slice. On the other hand, it is known that there exist knots with the same traces but different $4$-genus~\cite{Piccirillo_shake_genus}. For an excellent survey about other related results, we refer to~\cite{thesis}. Adapting the above notation, we say that two non-isotopic knots are $4$-\textit{dimensional friends} if their traces are orientation-preservingly diffeomorphic.

One interest in knots with the same $0$-surgeries is that such knots provide a potential way to disprove the smooth $4$-dimensional Poincar\'e conjecture (see for example~\cite{Freedman_Gompf_Man_Machine,Manolescu_Piccirillo_0_surgery}): If there exist friends $K$ and $K'$ such that $K$ is slice (i.e.\ bounds a smooth disk in $D^4$) but $K'$ is not slice, then the smooth $4$-dimensional Poincar\'e conjecture is false. If such a pair of knots exists we might expect it to be simple in some way. 

One obvious complexity notion for friends $K$ and $K'$ would be the sum of the crossing numbers $c(K)+c(K')$. In~\cite{Manolescu_Piccirillo_0_surgery} a systematic approach for constructing friends with relatively small crossing numbers was explored. In~\cite[Example~4.10]{Manolescu_Piccirillo_0_surgery} it was shown that $K12n309$ and $-K14n14254$ are friends and the authors briefly raise the question if these friends minimize $c(K)+c(K')$ among all friends $(K,K')$.\footnote{For knots of at most $15$ crossings we use the DT notation~\cite{DT_notation_and_table}. For knots with crossing numbers in $[16,19]$, we use Burton's notation~\cite{Burton}. We denote by $-K$ the mirror of $K$. Here (and in the rest of the paper) we only list one chirality of the friends. But of course if $(K,K')$ are friends, then $(-K,-K')$ are also friends.} The construction method in~\cite{Manolescu_Piccirillo_0_surgery} is based on RBG links, cf.~\cite{Miller_Piccirillo_traces}. Another successful method for constructing friends is annulus twisting~\cite{Osoinach_annulus}, cf.~\cite{AJDLO_same_traces,AbeTagami_annulus_presentation}. In fact, Figure 2 of~\cite{AbeTagami_annulus_presentation} shows that $K6a1$ and $19nh\_78$ are friends that yield a smaller sum of crossing numbers than the friends $K12n309$ and $-K14n14254$. 

Our main result demonstrates that the above friends $(K6a1,19nh\_78)$ realize the minimum among all friends. 

\begin{thm}\label{thm:cc}
The minimum value of $c(K)+c(K')$ among all friends $(K,K')$ is $25$. This minimum is uniquely realized by the pair $(K6a1,19nh\_78)$. 
\end{thm}

Another (arguably more natural) complexity notion for a knot $K$ is the \textit{tetrahedral complexity} $t(K)$, which is defined to be the minimal number of tetrahedra needed to ideally triangulate the complement of $K$. Hyperbolic knots with tetrahedral complexity at most $9$ have been enumerated by Dunfield~\cite{Dunfield_census}. These knots are called the \textit{census knots}.

\begin{thm}\label{thm:tt}
    The minimum value of $t(K)+t(K')$ among all hyperbolic friends $(K,K')$ is $12$. This minimum is uniquely realized by $(m224,-v3093)$.\footnote{Here we use the SnapPy census notation of these knots~\cite{Dunfield_census}. The first letter yields the tetrahedral complexity: $m$, $s$, $v$, $t$, $o9$, stand for $5$ (or less), $6$, $7$, $8$, and $9$ tetrahedra, respectively. The census manifold $m224$ is the complement of $-K11n38$, while $v3093$ is the complement of $-16nh\_9$.}
\end{thm}

Both pairs of minimizers $(K6a1,19nh\_78)$ and $(m224,-v3093)$ are also $4$-dimen\-sio\-nal friends. On the other hand, we can also ask for the simplest example of friends that are not $4$-dimensional friends. For the crossing number complexity, we can answer this question, while it remains open for the tetrahedral complexity.

\begin{thm}\label{thm:traces}
 The minimum value of $c(K)+c(K')$ among all friends $(K,K')$ that are not $4$-dimensional friends is $26$. This minimum is realized by two pairs of friends: $(K12n309,-K14n14254)$ and $(K10n10,-16nh\_17)$. 
    Any other potential minimizer is either $(K11n49,K15n103488)$ or a pair of friends consisting of a $6$-crossing and a $21$-crossing knot.
\end{thm}

For the proofs of Theorem~\ref{thm:cc} and~\ref{thm:tt} we use the following strategy. We have an upper bound on these complexity notions coming from the friends $K6a1$ and $19nh\_78$. Since the decision problem for compact $3$-manifolds with boundary and closed $3$-manifolds is solved~\cite{Kuperberg_decision_problem}, we can enumerate all pairs of knots up to that complexity and check if their $0$-surgeries are diffeomorphic. Here the problem is that the decision problem for $3$-manifolds is only partially implemented and that the enumeration of all these pairs has a high complexity. Nevertheless, we were able to create a census of pairs of knots that share the same $0$-surgery that contains all friends $(K,K')$ of low complexity. To perform the above strategy we have to use several short-cuts. For our proofs, we are using the verified functions from SnapPy~\cite{SnapPy} and Regina~\cite{Regina}, together with code and data from~\cite{ABG+19,CensusKnotInfo, Baker_Kegel_braid_positive,Burton,Dunfield_census,GAP4,Ribbon_ML,KnotInfo,KLO,Szabo_calculator,sagemath,Thistlethwaite,FPS_code,BKMa,BKMb}.

\begin{thm}\label{thm:census}
    The census of $41$ friends displayed in Table~\ref{tab:doubles} contains all friends $(K,K')$ with 
    \begin{itemize}
        \item $c(K),c(K')\leq15$, or
        \item $K$, $K'$ hyperbolic, and $t(K),t(K')\leq9$, or
        \item $c(K)+c(K')\leq25$, or 
        \item $K$, $K'$ hyperbolic, and $t(K)+t(K')\leq12$, or
        \item $K'$ hyperbolic, $c(K)\leq15$ and $t(K')\leq9$. 
    \end{itemize}
    Moreover, for $26$ of the $41$ pairs we determine if their traces are homeomorphic/diffeo\-mor\-phic or not.
\end{thm}

\begin{table}[htbp] 
	\caption{Each row represents a pair of friends. The column \textit{DT/B name} contains the DT name, its Burton name, or an interval in which the crossing number lies. The column \textit{census name} contains the SnapPy census name or an interval in which the tetrahedral complexity lies.}
	\label{tab:doubles}
 {\scriptsize
\begin{tabular}{|c|c|c||c|c|c||c|}
\hline
DT/B name & census name  & $g_4$  & DT/B name & census name & $g_4$  & traces \\  
\hline
\hline
$K6a1$ & $s912$ & $1$  &$19nh\_78$ & $[10,11]$ & $1$  & $C^\infty$  \\  
\hline
$K9n4$ & $-m199$ & $1$  & $-18nh\_23$ & $o9\_34801$ & $1$ & $C^\infty$ \\
\hline

$K10n10$ & $-t12200$ & $0$  & $-16nh\_17$ & $-t11532$ & $0$  & not $C^0$  \\  
\hline
$K10n10$ & $-t12200$ & $0$ & $[17,28]$ & $o9\_43446$ & $0$  & $C^\infty$ \\ 
\hline
$K10n13$ & $m201$ & $1$ & $[17,23]$ & $o9\_34818$ & $1$ & $C^\infty$ \\  	
\hline

$K11n38$ & $-m224$ & $1$ & $-16nh\_9$ & $v3093$ & $1$  & $C^\infty$ \\  			
\hline
$K11n49$ & $-v3536$ & $0$  &$K15n103488$ & $-v3423$ & $0$  &  \\  
\hline
$K11n49$ & $-v3536$ & $0$ & $[17,27]$ & $o9\_42735$ & $0$ & $C^\infty$ \\
\hline
$K11n116$ & $-t12748$ & $0$  & $17nh\_28$ & $-t12607$ & $0$ & $C^\infty$ \\ 		
\hline	

$K12n121$ & $s239$ & $1$&$[17,25]$ & $-t11441$ & $1$ & \\ 
\hline
$K12n200$ & $t09735$ & $1$  &$[17,23]$ & $t11748$ & $1$  & \\  
\hline
$K12n309$ & $-v3195$ & $0$  &$-K14n14254$ & $v2508$ & $0$& not $C^0$ \\  					
\hline
$K12n318$ & $-o9\_40519$ & $0$  & $-18nh\_32$ & $-o9\_39433$ & $0$ &  not $C^0$ \\  	
\hline

$K13n469$ & $v2272$ & $0$  &$-K13n469$ & $-v2272$ & $0$ & $C^\infty$ \\  					
\hline
$K13n572$ & $10$ & $1$  & $-K15n89587$ & $[10,11]$ & $1$  & $C^\infty$ \\  					
\hline
$K13n1021$ & $-t09900$ & $1$  & $[17,28]$ & $-o9\_34908$ & $1$   & \\ 
\hline
$K13n1021$ & $-t09900$ & $1$  & $K15n101402$ & $-o9\_43876$ & $1$&  \\
\hline
$K13n2527$ & $10$ & $1$  &$-K15n9379$ & $t12270$ & $1$ & $C^\infty$ \\  					
\hline
$K13n3158$ & $o9\_42515$ & $0$ & $19nh\_40$ & $o9\_41909$ & $0$&  \\  	
\hline

$K14n3155$ & $10$ & $1$ & $-K14n3155$ & $10$ & $1$&  \\  				
\hline
$K14n3155$ & $10$ & $1$  &$[16,20]$ & $t11462$ & $1$ & \\  				
\hline
$K14n3155$ & $10$ & $1$  &$[16,20]$ & $-t11462$ & $1$  & \\  				
\hline
$K14n3611$ & $o9\_33568$ & $1$  & $[16,25]$ & $-o9\_27542$ & $1$ &$C^\infty$ \\ 				
\hline	
$K14n5084$ & $-o9\_33833$ & $1$ & $[23,35]$ & $-o9\_37547$ & $1$ & $C^0$ \\ 				
\hline

$K15n19499$ & $o9\_37768$ & $0$  &$K15n153789$& $10$ & $0$ & $C^\infty$ \\ 
\hline
$K15n19499$ & $o9\_37768$ & $0$   & $[17,26]$ & $o9\_31828$ & $0$ & $C^\infty$ \\ 
\hline
$K15n94464$ & $[10,11]$ & $1$  &$[17,24]$ & $-o9\_40081$ & $1$ & $C^\infty$ \\  	
\hline
$K15n101402$ & $-o9\_43876$ & $1$  & $[17,28]$ & $-o9\_34908$ & $1$ &  \\ 
\hline
$K15n103488$ & $-v3423$ & $0$  &$[17,27]$ & $o9\_42735$ & $0$ &  \\  
\hline
$K15n153789$& $10$ & $0$  &$[17,26]$ & $o9\_31828$ & $0$& $C^\infty$ \\ 
\hline

$16nh\_17$ & $t11532$ & $0$  & $[17,28]$ & $-o9\_43446$ & $0$ & not $C^0$ \\ 
\hline

$18nh\_2$ & $v0595$ & $1$  & $[19,30]$ & $t12120$ & $1$ &  \\ 	
\hline
$18nh\_7$ & $v2869$ & $1$  & $[16,22]$ & $t12388$ & $1$ & $C^0$ \\  	
\hline	
$18nh\_16$ & $-t10974$ & $0$  & $[17,32]$ & $o9\_39967$ & $0$  &   \\  				
\hline

$19nh\_4$ & $t07281$ & $0$  & $[23,42]$ & $o9\_34949$ & $0$ & $C^\infty$ \\  
\hline
$19nh\_4$ & $t07281$ & $0$& $[17,31]$ & $o9\_39806$ & $0$ &  \\
\hline

$[16,20]$ & $t11462$ & $1$ &$[16,20]$ & $-t11462$ & $1$  &$C^\infty$\\  				
\hline
$[17,24]$ & $t11900$ & $1$  &$[16,28]$ & $-o9\_40803$ & $1$  & $C^0$\\  					
\hline
$[17,31]$ & $o9\_22951$ & $[0,1]$  &$[17,31]$ & $-o9\_22951$ & $[0,1]$ & $C^\infty$ \\  					
\hline
$[17,31]$ & $o9\_39806$ & $0$ & $[23,42]$ & $o9\_34949$ & $0$  &\\ 
\hline
$[16,26]$ & $o9\_41058$ & $[0,1]$ & $[16,26]$ & $-o9\_41058$ & $[0,1]$  &$C^\infty$\\  					
\hline
\end{tabular}
}
\end{table}

Theorems~\ref{thm:cc}, \ref{thm:tt} and~\ref{thm:traces} follow directly from Theorem~\ref{thm:census}.
In Table~\ref{tab:doubles} we also list the smooth $4$-genus $g_4$ of these knots. In particular, it turns out that all friends in our census share the same smooth $4$-genus. (They are either \textit{slice}, i.e.\ have smooth $4$-genus $0$, or have smooth $4$-genus $1$.) In particular, this rules out a large class of potential candidates for pairs
of knots with the same $0$-surgery that could be used to construct a
counterexample to the smooth $4$-dimensional Poincaré conjecture. If a pair of friends leading to a counterexample to the smooth
$4$-dimensional Poincaré conjecture exists, at least one of them needs to have large crossing number (at least $16$) and tetrahedral complexity (at least $10$). Thus, a brute force approach like that done in this article is not very likely to find such a pair of friends. Another computational strategy to search for a pair of friends with different sliceness statuses is explained in~\cite{Dunfield_exterior_to_link} and studied further in~\cite{Dunfield_Gong}. 

We also remark that all friends from our census share the same $\tau$-invariant from Heegaard Floer homology (which is known to not be a trace invariant) and the same  $\nu$-invariant (which is determined by the trace)~\cite{nu_trace}.

\subsection{Code and data}

The code and additional data accompanying this paper can be accessed at the GitHub page~\cite{GitHub}.

\subsection*{Acknowledgements}
We are happy to thank Ben Burton and Jonathan Spreer for explaining how to search through the Pachner graph with Regina, L\'eo Mousseau for help with the $4$-genus computations, Lisa Piccirillo for useful discussions and comments, Marithania Silvero for pointing us to~\cite{Ribbon_ML} and~\cite{Thistlethwaite}, and Morwen Thistlethwaite for making his data of the prime $20$ crossing knots available to us. Moreover, we thank the referee for careful proofreading.

\subsection*{Individual grant support}
T.A.\ is supported by the Research Promotion Program for Acquiring Grants in-Aid for Scientific Research (KAKENHI) in Ritsumeikan University. 

MK is supported by the DFG, German Research Foundation, (Project: 561898308) and by the SFB/TRR 191 \textit{Symplectic Structures in Geometry, Algebra and Dynamics}; by a Ram\'on y Cajal grant (RYC2023-043251-I) and the project PID2024-157173NB-I00 funded by MCIN/AEI/10.13039/50110001\-1033, by ESF+, and by FEDER, EU; and by a VII Plan Propio de Investigación y Transferencia (SOL2025-36103) of the University of Sevilla.

\section{The low-crossing knots}\label{sec:low}
In this section, we work with the low-crossing knots, i.e.\ the prime knots that have diagrams of at most $15$ crossings.
We start with the classification of their $0$-surgeries. 

\begin{thm}\label{thm:classification}
    Tables~\ref{tab:non_hyp_surgeries} and~\ref{tab:non_hyp_surgeries2} list the Regina names of all non-hyperbolic $0$-surgeries of prime knots with at most $15$ crossings. Any knot that is not contained in that table is proven to have a hyperbolic $0$-surgery. The list of the hyperbolic volumes of these $0$-surgeries can be accessed at~\cite{GitHub}.
\end{thm}

The complete list of prime, non-hyperbolic knots of at most $15$ crossings and their descriptions as torus knots or satellite knots is given in Table~\ref{tab:non_hyp_knots}. Any $0$-surgery on such a knot is again non-hyperbolic.

\begin{proof}[Proof of Theorem~\ref{thm:classification}]
SnapPy has implemented a verified version of the volume function, that if successful, returns an interval of small length with rational boundaries that is guaranteed to contain
the exact value of the hyperbolic volume. In that case, the manifold is also proven to be hyperbolic. If unsuccessful, an error is returned. In that case, the manifold might be hyperbolic or not.

We use that verified volume function in SnapPy to compute for each $0$-surgery a small interval in which the exact value of the volume lies. If SnapPy fails to compute the verified volume, we try to find a better triangulation where it works or raise the precision for the computations. If both methods fail to compute the volume, we use the Regina recognition code~\cite{Dunfield_census}, cf.~\cite{FPS_code} to recognize the triangulation of the $0$-surgery as a triangulation of a non-hyperbolic manifold. This successfully classifies all $0$-surgeries of the low-crossing knots. 
\end{proof}

\begin{table}[htbp] 
	\caption{This table collects all non-hyperbolic $0$-surgeries along knots with at most $15$ crossings, together with their Regina names \cite{Dunfield_census}. Here $A$ stands for an annulus, $P$ for a pair of pants, and $M$ for a M\"obius strip. If $X$ is a census manifold with $2$-boundary components, then $JSJ(X)$ denotes a closed manifold obtained by gluing the boundaries of $X$ together with some diffeomorphism.}
	\label{tab:non_hyp_surgeries}
\begin{tabular}{|c|c|}
\hline
DT name & non-hyperbolic name \\  
\hline
 $K3a1$ & $SFS \,[S^2: (2,1) (3,1) (6,-5)]$  \\ 
 $K4a1$ & $T^2 \times I / [ 2,1 | 1,1 ]$  \\ 
 $K5a1$ & $SFS \,[A: (2,1)] / [ 0,1 | 1,-1 ]$  \\ 
 $K5a2$ & $SFS \,[S^2: (2,1) (5,2) (10,-9)]$  \\ 
 $K6a3$ & $SFS \,[A: (2,1)] / [ 0,1 | 1,-2 ]$  \\ 
 $K7a4$ & $SFS \,[A: (3,2)] / [ 0,1 | 1,-1 ]$  \\ 
 $K7a6$ & $JSJ\big(SFS \,[A: (2,1)] \cup SFS \,[A: (2,1)]\big)$  \\ 
 $K7a7$ & $SFS \,[S^2: (2,1) (7,3) (14,-13)]$  \\ 
 $K8a11$ & $SFS \,[A: (3,1)] / [ 0,1 | 1,-2 ]$  \\ 
 $K8a18$ & $JSJ\big(SFS \,[A: (2,1)]\cup SFS \,[A: (2,1)]\big)$  \\
 $K8n1$ & $SFS \,[D^2: (2,1) (2,1)] \cup_m SFS \,[D^2: (3,1) (3,2)], m = [ 0,1 | 1,0 ]$  \\ 
 $K8n3$ & $SFS \,[S^2: (3,2) (4,1) (12,-11)]$  \\ 
   
 $K9a27$ & $SFS \,[A: (4,3)] / [ 0,1 | 1,-1 ]$  \\ 
 $K9a36$ & $JSJ\big(SFS \,[A: (2,3)]\cup SFS \,[A: (3,1)]\big) $  \\ 
 $K9a40$ & $JSJ\big(m202)$  \\ 
 $K9a41$ & $SFS \,[S^2: (2,1) (9,4) (18,-17)]$  \\ 
 $K9n5$ & $SFS \,[A: (2,1)] / [ -1,3 | 1,-2 ]$  \\ 

 $K10a75$ & $SFS \,[A: (4,1)] / [ 0,1 | 1,-2 ]$  \\ 
 $K10a117$ & $JSJ\big(SFS \,[A: (2,1)]\cup SFS \,[A: (3,5)]\big) $  \\ 
  $K10n13$ & $SFS \,[S^2: (2,1) (5,1) (10,-7)]$  \\ 
 $K10n21$ & $SFS \,[S^2: (3,1) (5,3) (15,-14)]$  \\ 
 $K10n29$ & $SFS \,[D^2: (2,1) (2,1)] \cup_m  SFS \,[D^2: (3,1) (3,2)], m = [ -1,2 | 0,1 ]$  \\ 
 
 $K11a247$ & $SFS \,[A: (5,4)] / [ 0,1 | 1,-1 ]$  \\ 
 $K11a343$ & $JSJ\big(SFS \,[A: (2,1)]\cup SFS \,[A: (4,7)]\big) )$  \\ 
 $K11a362$ & $JSJ\big(m357\big)$  \\ 
 $K11a363$ & $JSJ\big(SFS \,[A: (3,1)]\cup SFS \,[A: (3,2)]\big) $  \\ 
 $K11a367$ & $SFS \,[S^2: (2,1) (11,5) (22,-21)]$  \\ 
 $K11n139$ & $SFS \,[A: (2,1)] / [ 2,5 | 1,2 ]$  \\ 
 $K11n141$ & $JSJ\big(m125\big)$  \\ 
 
 $K12a803$ & $SFS \,[A: (5,1)] / [ 0,1 | 1,-2 ]$  \\ 
 $K12a1166$ & $JSJ\big(SFS \,[A: (2,1)]\cup SFS \,[A: (4,1)]\big) )$  \\ 
 $K12a1287$ & $JSJ\big(SFS \,[A: (3,1)]\cup SFS \,[A: (3,4)]\big) $  \\ 
 $K12n121$ & $SFS \,[D^2: (2,1) (2,1)] \cup_m  SFS \,[D^2: (2,1) (3,1)], m = [ 5,1 | 4,1 ]$  \\ 
 $K12n582$ & $SFS \,[D^2: (2,1) (2,1)] \cup_m  SFS \,[D^2: (3,1) (3,2)], m = [ -2,3 | -1,2 ]$  \\ 
 $K12n721$ & $JSJ\big(SFS \,[D^2: (2,1) (2,-1)]\cup m043\big)$  \\ 
 
 $K13a3143$ & $SFS \,[A: (6,5)] / [ 0,1 | 1,-1 ]$  \\ 
  $K13a4573$ & $JSJ\big(SFS \,[A: (2,1)]\cup SFS \,[A: (5,4)]\big) $  \\ 
 $K13a4843$ & $JSJ\big(s548\big)$  \\ 
 $K13a4856$ & $JSJ\big(SFS \,[A: (3,1)]\cup SFS \,[A: (4,3)]\big) $  \\ 
 $K13a4873$ & $JSJ\big(s876\big)$  \\ 
 $K13a4878$ & $SFS \,[S^2: (2,1) (13,6) (26,-25)]$  \\ 
 
 $K13n469$ & $JSJ\big(SFS \,[D^2: (2,1) (2,-1)]\cup m004\big)$  \\ 

 $K13n3521$ & $JSJ\big(m329\big)$  \\ 
 $K13n3523$ & $SFS \,[A: (2,1)] / [ 2,7 | 1,3 ]$  \\ 
 
\hline
 
\hline
\end{tabular}
\end{table}

\begin{table}[htbp] 
	\caption{Table~\ref{tab:non_hyp_surgeries} continued.}
	\label{tab:non_hyp_surgeries2}
\begin{tabular}{|c|c|}
\hline
DT name & non-hyperbolic name \\  
\hline
 $K13n3594$ & $JSJ\big(m292\big)$  \\ 

 $K13n3596$ & $JSJ\big(SFS \,[D^2: (2,1) (3,-2)]\cup SFS \,[P: (1,3)]\big) $  \\ 

 $K13n3663$ & $JSJ\big(SFS \,[D^2: (2,1) (3,-2)]\cup SFS \,[P: (1,1)]\big) $  \\ 

 $K13n4587$ & $SFS \,[D^2: (2,1) (3,1)] \cup_m  SFS \,[D^2: (2,1) (14,1)], m = [ -3,5 | -1,2 ]$  \\ 
 $K13n4639$ & $SFS \,[D^2: (2,1) (3,1)] \cup_m  SFS \,[D^2: (2,1) (10,1)], m = [ -5,7 | -2,3 ]$  \\ 

$K14a12741$ & $SFS \,[A: (6,1)] / [ 0,1 | 1,-2 ]$  \\ 
 $K14a17730$ & $JSJ\big(SFS \,[A: (2,1)]\cup SFS \,[A: (5,9)]\big) $  \\ 
  $K14a19429$ & $JSJ\big(SFS \,[A: (3,4)]\cup SFS \,[A: (4,1)]\big) $  \\ 
  
 $K14n3611$ & $JSJ\big(M \Tilde\times S^1\cup  m015 \big)$  \\ 
 $K14n18212$ & $SFS \,[D^2: (2,1) (2,1)] \cup_m  SFS \,[D^2: (3,1) (3,2)], m = [ -3,4 | -2,3 ]$  \\ 
 $K14n19265$ & $JSJ\big(v3319\big)$  \\ 

 $K14n21881$ & $SFS \,[S^2: (3,2) (7,2) (21,-20)]$  \\ 
 $K14n21882$ & $SFS \,[D^2: (2,1) (2,1)] \cup_m  SFS \,[D^2: (3,1) (6,1)], m = [ 1,1 | 0,1 ]$  \\ 
 $K14n22073$ & $SFS \,[D^2: (2,1) (3,1)] \cup_m  SFS \,[A: (2,1) (2,1)] \cup_n SFS \,[D^2: (2,1) (3,2)]$,\\
 &$ m = [ 1,-1 | 0,-1 ], n = [ 0,1 | 1,1 ]$  \\ 
 $K14n22180$ & $JSJ\big(SFS \,[D^2: (2,1) (3,-2)]\cup SFS \,[$P$: (1,1)]\big) $  \\ 
 $K14n22185$ & $JSJ\big(SFS \,[D^2: (2,1) (2,-1)]\cup  m137\big)$  \\ 
 $K14n22589$ & $JSJ\big(m129\big)$  \\ 
 $K14n24553$ & $JSJ\big(SFS \,[D^2: (2,1) (2,-1)]\cup  s663\big)$  \\ 
 $K14n26039$ & $JSJ\big(P \times  S^1\cup SFS \,[D^2: (2,1) (3,-2)]\big) $  \\ 

 $K15a54894$ & $SFS \,[A: (7,6)] / [ 0,1 | 1,-1 ]$  \\ 
 $K15a78880$ & $JSJ\big(SFS \,[A: (2,1)]\cup SFS \,[A: (6,5)]\big) $  \\ 
 $K15a84844$ & $JSJ\big( v1203\big)$  \\ 
 $K15a84969$ & $JSJ\big(SFS \,[A: (3,1)]\cup SFS \,[A: (5,4)]\big) $  \\ 
 $K15a85213$ & $JSJ\big( v2601\big)$  \\ 
 $K15a85234$ & $JSJ\big(SFS \,[A: (4,1)]\cup SFS \,[A: (4,7)]\big) $  \\ 
 $K15a85257$ & $JSJ\big( v3461\big)$  \\ 
 $K15a85263$ & $SFS \,[S^2: (2,1) (15,7) (30,-29)]$  \\

 $K15n19499$ & $JSJ\big(SFS \,[D^2: (2,1) (2,-1)]\cup  m032\big)$  \\ 
 $K15n40211$ & $SFS \,[D^2: (2,1) (3,1)] \cup_m  SFS \,[D^2: (2,1) (18,1)], m = [ -1,3 | 0,1 ]$  \\ 
 $K15n41185$ & $SFS \,[S^2: (4,3) (5,1) (20,-19)]$  \\ 
 $K15n43522$ & $SFS \,[A: (2,1)] / [ 3,11 | 2,7 ]$  \\ 
 $K15n48968$ & $JSJ\big(M\Tilde\times S^1\cup  v2817\big)$  \\ 
 $K15n51748$ & $JSJ\big(SFS \,[D^2: (2,1) (2,-1)]\cup m137\big)$  \\ 
 
 $K15n59184$ & $JSJ\big(P\times S^1\cup SFS \,[D^2: (2,1) (3,-2)]\big) $  \\ 
 $K15n72303$ & $JSJ\big(M\Tilde\times S^1\cup  s493\big)$  \\ 
 
 $K15n112477$ & $JSJ\big( s503\big)$  \\ 
 $K15n112479$ & $SFS \,[A: (2,1)] / [ 2,9 | 1,4 ]$  \\ 
 $K15n113773$ & $JSJ\big( s843\big)$  \\ 
 $K15n113775$ & $JSJ\big( m129\big)$  \\ 
 $K15n113923$ & $JSJ\big( s441\big)$  \\ 
 $K15n115375$ & $JSJ\big( m129\big)$  \\ 
 $K15n115646$ & $JSJ\big(SFS \,[D^2: (2,1) (3,-2)]\cup SFS \,[P: (1,1)]\big) $  \\ 
 $K15n124802$ & $SFS \,[D^2: (2,1) (3,1)] \cup_m  SFS \,[D^2: (2,1) (6,1)], m = [ -7,9 | -3,4 ]$  \\ 
 $K15n142188$ & $JSJ\big(SFS \,[D^2: (2,1) (3,-2)]\cup 1\textrm{-cusped hyperbolic}\big)$  \\ 
 $K15n153789$ & $JSJ\big(SFS \,[D^2: (2,1) (2,-1)]\cup  m032\big)$  \\ 
 $K15n156076$ & $JSJ\big(SFS \,[D^2: (2,1) (3,-2)]\cup 1\textrm{-cusped hyperbolic}\big)$  \\ 
 $K15n160926$ & $JSJ\big(t11128\big)$  \\ 
 $K15n164338$ & $JSJ\big( s906\big)$\\
\hline
 
\hline
\end{tabular}
\end{table}

Next, we create a census of all friends $(K,K')$ such that $K$ and $K'$ are both prime and have at most $15$ crossings. It is known that knots with diffeomorphic $0$-surgeries share the same Alexander polynomial. So we first run through the list of all prime knots with at most $15$ crossings, compute their Alexander polynomials, and sort the knots into groups that have the same Alexander polynomials. If it turns out that some group contains only a single knot $K$, then this knot cannot have a friend that is prime and that has at most $15$ crossings, and we do not consider $K$ anymore. (The possibility that $K$ and $-K$ are not-isotopic but still are friends is discussed in the proof of Proposition~\ref{prop:mirror} below.) 

So in the following, we only consider groups of knots with the same Alexander polynomials that contain more than one knot. In the next step, we use the results from Theorem~\ref{thm:classification} to further split up the groups into knots that have all non-hyperbolic or all hyperbolic $0$-surgeries. If the $0$-surgery is hyperbolic, we also split into groups of knots where the intervals, containing the volumes of the $0$-surgeries, overlap. By using the verified symmetry group computations in SnapPy we can further distinguish a few more $0$-surgeries.

Some of the non-hyperbolic $0$-surgeries we can distinguish by their Regina names \cite{Dunfield_census}. For example, using the pieces in their JSJ decomposition and the classification of Seifert fibered spaces. Note, however, that the gluing map of the JSJ pieces is not returned. So if two Regina names agree or disagree, that does not necessarily mean that the manifolds are diffeomorphic or non-diffeomorphic.

Next, we know that several other knot invariants agree for knots that share the same $0$-surgeries. If two knots share the same $0$-surgery, they have the same $3$-genus, the same fiberedness status~\cite{Gabai}, the same signatures, and isomorphic knot Floer homologies groups in certain gradings~\cite[Corollary~4.5]{Ozsvath_Szabo_holomorphic}, cf.~\cite[Remark~2.2]{Baldwin_Sivek_0_inf}. We use Szab\'o's knot Floer homology calculator~\cite{Szabo_calculator} to compute these invariants and refine our groups further.

For the remaining groups of knots that might share the same $0$-surgeries, we try to distinguish their $0$-surgeries by showing that their fundamental groups are not isomorphic, as done for example in~\cite{Burton,Dunfield_Lspace}. If $H$ is a subgroup of a fundamental group $G$ of a $0$-surgery and $N$ the core of $H$, then the tuple 
\begin{equation*}
    \big([G:H],[G:N],H^{ab}, N^{ab}\big)
\end{equation*}
is an invariant of the subgroup $H < G$. We use SnapPy, Regina, and GAP~\cite{GAP4} to determine all conjugacy classes of subgroups of $G$ up to index $7$ and compute for each the above invariants. If the collections of these invariants do not agree, then the fundamental groups are not isomorphic and thus the $0$-surgeries are not diffeomorphic.

This was enough to distinguish the $0$-surgery of most pairs of low-crossing knots. However, for around $200$ pairs of knots these invariants either all agreed or were too time-consuming to compute. The remaining pairs we could distinguish using their length spectra. The verified computations of the length spectra in SnapPy up to length $4$ were enough to distinguish all but $6$ pairs.

Among the remaining $6$ pairs of knots, we check if the $0$-surgeries are diffeomorphic. For hyperbolic $0$-surgeries this is done with SnapPy by searching for an isometry. For the non-hyperbolic $0$-surgeries we load the triangulations to Regina and search through the Pachner graphs for a combinatorial equivalence between the triangulations.

Note that the SnapPy search for isometries between hyperbolic $3$-manifold also returns orientation-reversing isometries. But we want to have orientation-preserving isometries. For that, we use the following method to search for orientation-preserving isometries. For closed manifolds, SnapPy does not give any information on whether the isometry is orientation-preserving or not. To determine if two closed $3$-manifolds $M_1$ and $M_2$ are orientation preserving-diffeomorphic, we can use SnapPy to drill out curves $c_1$ and $c_2$ from $M_1$ and $M_2$. For cusped manifolds, SnapPy also displays the action of the isometry on the cusp from which one can read off whether the isometry is orientation-preserving or not. Then we can search for an orientation-preserving isometry from $M_1\setminus c_1$ to $M_2\setminus c_2$ that fixes the meridians of $c_1$ and $c_2$. Such an isometry extends uniquely over the Dehn fillings to an orientation-preserving isometry
$M_1 \to M_2$, and by construction this extension maps the core curve $c_1$ to $c_2$.

In total, we find $6$ pairs of friends among the low-crossing knots (and below two more pairs, each consisting of a knot and its mirror). These are displayed in Table~\ref{tab:doubles}. All other pairs of prime knots with at most $15$ crossings have non-diffeomorphic $0$-surgeries. 

\section{Connected sums}

Next, we show that no knot among the low crossing knots has $0$-surgery equal to the $0$-surgery of a connected sum.
For that, we first prove the following general lemma.

 \begin{lem} \label{lem:0_sum}
Let $K=K_1\#\cdots\#K_n$ be a connected sum of non-trivial, non-satellite knots $K_1,\ldots,K_n$. If $K'$ is another non-trivial connected sum such that $K$ and $K'$ have orientation-preserving diffeomorphic $0$-surgeries, then $K$ and $K'$ are isotopic.
 \end{lem}

\begin{proof}
It is well-known (see for example~\cite{Budney_JSJ}) that the exterior $E_K$ of $K$ decomposes as
\begin{equation*}
    (D_n \times S^1) \cup \bigcup_{i=1}^n E_{K_i}
\end{equation*}
where $D_n$ denotes a $2$-disk with $n$ holes. The exterior $E_{K_i}$ is glued to the $i$-th hole with a gluing map that identifies the meridian $\mu_i$ of $K_i$ with the boundary of the hole and the Seifert longitude $\lambda_i$ of $K_i$ with the $S^1$-factor. Then the longitude $\lambda_K$ of $K$ is given by the outer boundary component of $D_n$ and the meridian $\mu_K$ of $K$ is the $S^1$-factor. Since every summand $K_i$ was assumed to not be a satellite, it follows that each $K_i$ is either a hyperbolic knot or a torus knot, and thus the above decomposition is actually the JSJ decomposition of $E_K$.

Now it follows that performing the $0$-surgery on $K$ is the same as capping off the outer boundary component of $D_n$ to get a $2$-sphere $S_n$ with $n$ holes. Thus the JSJ decomposition of the $0$-surgery of $K$ is given as
\begin{equation*}
    K(0)=(S_n \times S^1) \cup \bigcup_{i=1}^n E_{K_i}
\end{equation*}
where each exterior $E_{K_i}$ is glued to the $i$-th hole with a gluing map that identifies the meridian $\mu_i$ of $K_i$ with the boundary of the hole and the Seifert longitude $\lambda_i$ of $K_i$ with the $S^1$-factor. (Note that in the case of $n=2$, the above JSJ decomposition simplifies to just $E_{K_1}\cup E_{K_2}$ such that meridians and longitudes get identified.)

Now let $K'=K'_1\#\cdots\#K'_{n'}$ be another connected sum of non-trivial, prime knots $K'_1,\ldots,K'_{n'}$ such that $K(0)$ is orientation-preserving diffeomorphic to $K'(0)$. Then the uniqueness of the JSJ decomposition implies that the JSJ decompositions of $K(0)$ and $K'(0)$ are the same, i.e.\
\begin{equation*}
    K(0)=(S_n \times S^1) \cup \bigcup_{i=1}^n E_{K_i}=K'(0).
\end{equation*}
On the other hand, we can decompose $K'(0)$ as
\begin{equation*}
    K'(0)=(S_{n'} \times S^1) \cup \bigcup_{i=1}^{\,n'} E_{K'_i}.
\end{equation*}
Note that the latter might not be the JSJ decomposition since $E_{K'_i}$ might have a non-trivial JSJ decomposition (this happens if $K'_i$ is a satellite knot). 

Next, we compare these two decompositions to show that $n=n'$ and (after renumbering) $E_{K_i}$ is orientation-preserving diffeomorphic to $E_{K'_i}$.

By assumption, each $E_{K_i}$ is either the exterior of a hyperbolic knot or a torus knot. In particular, no $E_{K_i}$ is a product. If $n=2$, then $K(0)$ contains no product and thus also $n'=2$. It follows that (after renumbering) $E_{K_i}$ is orientation-preserving diffeomorphic to $E_{K'_i}$.

If $n>2$, then the JSJ decomposition of $K(0)=K'(0)$ contains a unique piece of the form $S_n\times S^1$, and thus we have two cases. The first case is that $n'=2$ and exactly one of the $E_{K'_i}$ contains a JSJ piece of the form $S_n\times S^1$. The second case is $n=n'$, then $S_n\times S^1$ is mapped to $S_{n'}\times S^1$ and (after renumbering) $E_{K_i}$ is orientation-preserving diffeomorphic to $E_{K'_i}$. We argue that the first case cannot occur. 
Indeed, suppose that $n'=2$ and that one of the exteriors $E_{K'_i}$ contains a JSJ piece diffeomorphic to $S_n\times S^1$. Since $K'_i$ is assumed to be prime, it has to be a satellite knot. Then, by the JSJ classification of knot exteriors due to Budney~\cite{Budney_JSJ}, the JSJ decomposition of
$E_{K'_i}$ contains a piece that is the exterior of a link in $S^3$ with at least two
components. On the other hand, since none of the knots $K_i$ is a satellite by assumption, the JSJ decomposition of $K(0)$ contains no JSJ piece other than $S_n \times S^1$ which has more than one boundary component.
This contradicts the uniqueness of the JSJ decomposition. Hence, the first case cannot occur.

 In summary we have $n=n'$ and (after renumbering) $E_{K_i}$ is orientation-preserving diffeomorphic to $E_{K'_i}$. Then~\cite{Gordon_Luecke} tells us that $K_i$ is isotopic to $K'_i$. 

Note, however, that the oriented diffeomorphism type of $E_{K_i}$ only determines the chirality of $K_i$, but does not yield an orientation of the knot $K_i$. On the other hand, the connected sum depends in general on the orientation of $K_i$. The orientations of the $K_i$ can be recovered from the gluing maps of the above JSJ decomposition (since the orientations of meridian and longitude change if the orientation of the knot is changed). Thus it follows that $K$ is isotopic to $K'$.
\end{proof}

From the proof of the above lemma, we get, in particular, the following decomposition theorem for the $0$-surgery of a connected sum.

\begin{cor}\label{cor:conneted_sum}
   Let $K=K_1\#\cdots\#K_n$ be a connected sum of non-trivial knots $K_i$, then the $0$-surgery $K(0)$ decomposes as
\begin{equation*}
    K(0)=(S_n \times S^1) \cup \bigcup_{i=1}^n E_{K_i},
\end{equation*}
   where $E_{K_i}$ is the exterior of $K_i$, $S_n$ denotes a $2$-sphere with $n$ holes, and where the gluing is along incompressible tori.

   In particular, for $n=2$, it follows that the $0$-surgery is obtained by gluing the exteriors of $K_1$ and $K_2$, i.e.\ $K(0)=E_{K_1}\cup E_{K_2}$.\qed
\end{cor}

With that corollary, we are ready to prove the following proposition which allows us to restrict to prime knots.

\begin{prop}\label{prop:prime}
Let $K$ be a prime knot with at most $15$ crossings or a hyperbolic knot with tetrahedral complexity at most $9$. Then any friend of $K$ is also prime.
\end{prop}

\begin{proof}
From Corollary~\ref{cor:conneted_sum} we know that the $0$-surgery of a non-trivial connected sum decomposes along incompressible tori as
\begin{equation*}
    (S_n \times S^1) \cup \bigcup_{i=1}^n E_{K_i}
\end{equation*}
or for $n=2$ as $E_{K_1}\cup E_{K_2}$.
However, from Tables~\ref{tab:non_hyp_surgeries} and~\ref{tab:non_hyp_surgeries2} we see directly that there are only two $0$-surgeries that might be of that form: The $0$-surgeries of $K_+=K15n142188$ and $K_-=K15n156076$ are both given by gluing the exterior $E_{K3a1}$ of the trefoil knot $K3a1$ to a $1$-cusped hyperbolic manifold $M_\pm$. To finish the proof, it is enough to show that in both cases the manifold $M_\pm$ is not the exterior of a knot in $S^3$. 

From Table~\ref{tab:non_hyp_knots} we see that $K_\pm$ are satellites of the trefoil $K3a1$ with pattern $\pm L6a1$. It follows that $M_\pm$ is given by $w_\pm$-surgery on one component\footnote{Since $L6a1$ has an isometry that interchanges the components (and preserves the framing) it does not matter which component we fill.} of $\pm L6a1$, where $w_\pm=12\pm 3$ is the writhe of the diagrams of $K_\pm$ shown in Figure~\ref{fig:knots}. From that description, we can build the manifold $M_\pm$ in SnapPy, compute the short slopes (i.e.\ those of length at most $6$) of $M_\pm$, and verify that the corresponding fillings are not homology spheres. Then the $6$-theorem~\cite{Ag00,La00} implies that $M_\pm$ is not the exterior of a knot in $S^3$.

For the census knots, all exceptional fillings are classified by Dunfield~\cite{Dunfield_census}. From his list, we can read off as above that no $0$-surgery can be a connected sum.
\end{proof}

\begin{table}[htbp] 
	\caption{The complete list of all non-hyperbolic prime knots up to $15$ crossings and their non-hyperbolic names.}
	\label{tab:non_hyp_knots}
\begin{tabular}{|c|c|}
\hline
DT name & non-hyperbolic name \\  
\hline
$K3a1$ & $(2,3)$-torus knot \\ 
$K5a2$ & $(2,5)$-torus knot \\ 
$K7a7$ & $(2,7)$-torus knot \\ 
$K8n3$ & $(3,4)$-torus knot \\ 
$K9a41$ & $(2,9)$-torus knot \\ 
$K10n21$ & $(3,5)$-torus knot \\ 
$K11a367$ & $(2,11)$-torus knot \\ 
$K13a4878$ & $(2,13)$-torus knot \\ 
$K13n4587$ & $(2,7)$-cable of $K3a1$ \\ 
$K13n4639$ & $(2,5)$-cable of $K3a1$ \\
$K14n21881$ & $(3,7)$-torus knot \\ 
$K14n22180$ & $(-3)$-fold twisted positive Whitehead double of $K3a1$ \\ 
$K14n26039$ & $(-3)$-fold twisted negative Whitehead double of $K3a1$ \\ 
$K15a85263$ & $(2,15)$-torus knot \\ 
$K15n40211$ & $(2,9)$-cable of $K3a1$ \\ 
$K15n41185$ & $(4,5)$-torus knot \\ 
$K15n59184$ & $(-2)$-fold twisted negative Whitehead double of $K3a1$ \\ 
$K15n115646$ & $(-4)$-fold twisted positive Whitehead double of $K3a1$ \\ 
$K15n124802$ & $(2,3)$-cable of $K3a1$ \\ 
$K15n142188$ & $(-3)$-fold twisted satellite of $K3a1$ with pattern $L6a1$, see Figure~\ref{fig:knots} \\ 
$K15n156076$ & $(-3)$-fold twisted satellite of $K3a1$ with pattern $-L6a1$, see Figure~\ref{fig:knots} \\ 
\hline
\end{tabular}
\end{table}

\begin{figure}[htbp]
 \centering
  \includegraphics[width=0.49\textwidth]{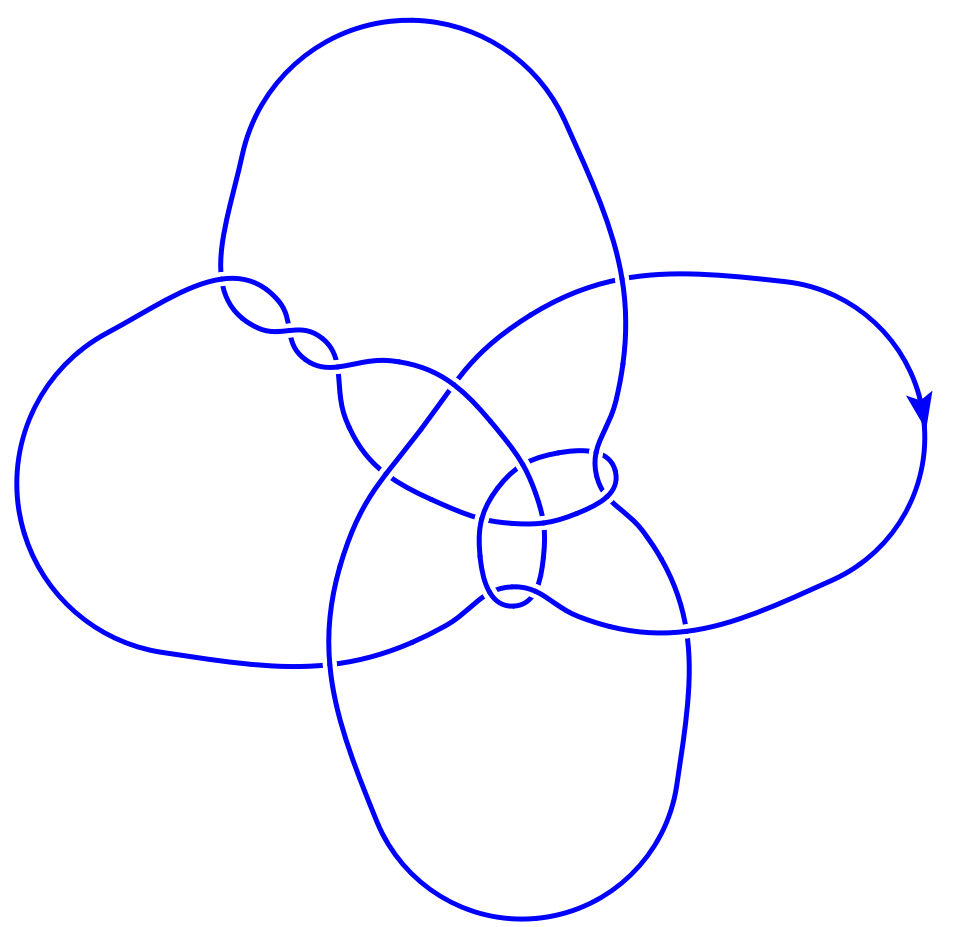}
 	\includegraphics[width=0.49\textwidth]{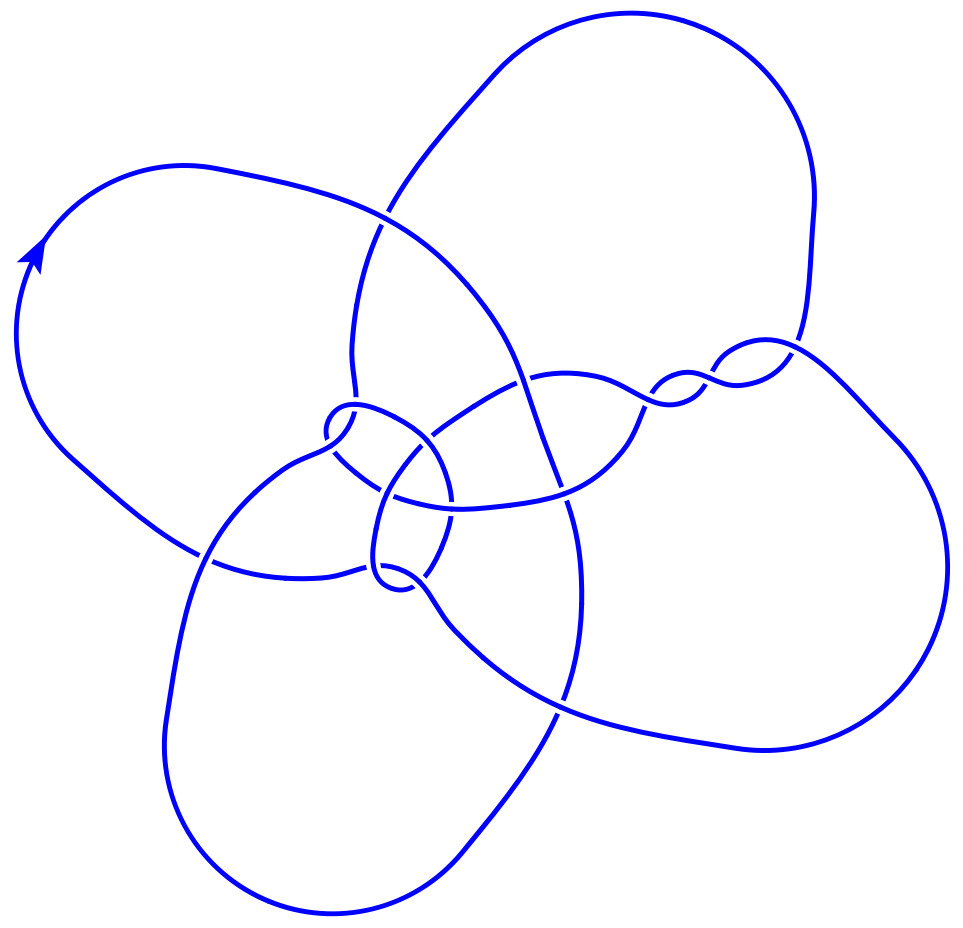}
 	\caption{Diagrams of $K15n142188$ (left) and $K15n156076$ (right) showing their satellite structures on $K3a1$. Both diagrams were created with KLO~\cite{KLO}.}
 	\label{fig:knots}
 \end{figure}

\section{Mirrors}

In this section, we study when mirrors can be friends. First, we describe an infinite family of knots that are friends with their mirrors.

\begin{prop}\label{prop:mirror_family}
    There exists an infinite family of mutually distinct knots $(K_m)_{m\in\N}$ such that for every $m\geq1$,
    \begin{itemize}
        \item The mirror $-K_m$ of $K_m$ is isotopic to $K_{-m}$,
        \item $K_m$ is not isotopic to its mirror $-K_m$, but
        \item the traces $X(K_m)$ and $X(-K_m)$ are orientation-preservingly diffeomorphic.
    \end{itemize}
\end{prop}

\begin{proof}
    We will describe the infinite family via an RBG link construction as for example in~\cite{Miller_Piccirillo_traces,Piccirillo_shake_genus,Piccirillo_Conway_knot,Manolescu_Piccirillo_0_surgery,concordant_traces}. For that we consider the $3$-component hyperbolic link $RBG$ with components $R$, $G$, and $B$ in Figure~\ref{fig:strongly_invertible_RBG}(i). By rotating that diagram around the $z$-axis, we see that there exists an involution $f\colon S^3\setminus RBG\rightarrow S^3\setminus RBG$, that maps $R$ to itself but interchanges $G$ and $B$. 

    Now we see $RBG$ as a Kirby diagram by dotting $R$ and adding framings $-m-1$ and $m-1$ to $G$ and $B$, respectively. Note that $R$ is a meridian of both $G$ and $B$ (but not simultaneously). Thus we can cancel $R$ and $G$ and denote the image of $B$ under this cancellation by $K_B^{-m}$. By symmetry, we can also reverse the roles of $B$ and $G$ and get a knot $K_G^{m}$. By construction, these knots share the same trace. From the above symmetry of the RBG link, it follows that $K_B^m$ is isotopic to $K_G^m$ and we define this knot to be $K_m$. Thus we have constructed knots $K_m$ such that $K_m$ and $K_{-m}$ share the same trace. Figure~\ref{fig:strongly_invertible_RBG}(ii)-(iii) show explicit Kirby moves to create out of the RBG link the knots $K_m$. 

    Next we show that $-K_m=K_{-m}$ and that the $K_m$ are mutually distinct. Since the $K_m$ are obtained from $K_0$ by twisting three strands $m$ times, it follows from~\cite{twisting} that infinitely many of the $K_m$ are distinct. More concretely, we can see that they are all mutually distinct. For that, we consider the surgery description of the $K_m$ along the $2$-component link $L$ shown in Figure~\ref{fig:strongly_invertible_RBG}(iv). With SnapPy we check that $L$ is a hyperbolic link. By computing the volumes of the $K_m$ for small values of $m$ and using~\cite{volumes} we conclude that the volumes of the $K_m$ converge monotonically to the volume of $L$ and in particular all $K_m$ (for $m\in\N$) have different volumes. 

    Moreover, we use SnapPy to compute the symmetry group of $L$, from which we observe that $L$ is isotopic to its mirror. This implies that $K_{-m}$ is isotopic to the mirror $-K_m$ of $K_m$.

    It remains to demonstrate the each of the $K_m$ is not isotopic to its mirror. This can also be deduced from the symmetry group of $L$. Using SnapPy, we check that the complement of $L$ admits no non-trivial orientation-preserving isometry that extends to the link. Thus Thursthon's hyperbolic Dehn filling theorem~\cite{Thurston} implies that any sufficiently large Dehn filling of $L$ yields manifolds without any isometry, see~\cite{BKMb} for details. The $K_m$ arise by Dehn filling one component of $L$ with slope $1/m$. Since the length of $1/m$ converges monotonically to $\infty$ it follows that infinitely many of the $K_m$ are asymmetric and thus not isotopic to their mirrors. More concretely, we can show that in fact for all $m\in\N$ the knots $K_m$ are not isotopic to their mirrors. For that we use the quantification of Thurston's hyperbolic Dehn filling theorem~\cite{Futer_Purcell_Schleimer_bound} (see for example~\cite{BKMb}) to compute an explicit bound $B$ such that $K_m$ is asymmetric for all $m\geq B$. For the finitely remaining $m$, we use SnapPy to create triangulations of the complements of $K_m$ and check that their symmetry groups do not contain any orientation-reversing diffeomorphisms. (But $K_1$ and $K_2$ turn out to be not asymmetric.)
\end{proof}

\begin{figure}[htbp] 
	\centering
	\def\svgwidth{0,87\columnwidth}
	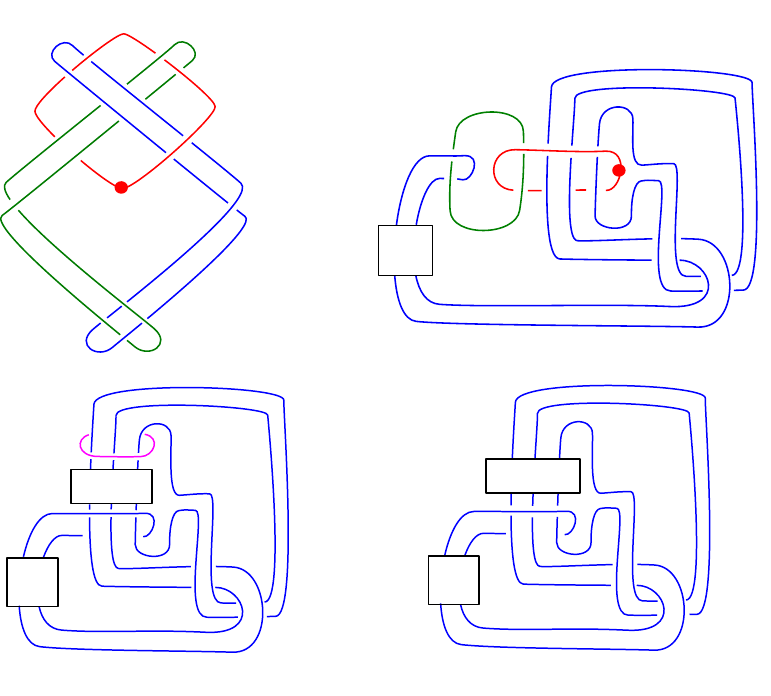
	\caption{A symmetric RBG link yielding knots $K_m$ that are friends with their mirrors}
	\label{fig:strongly_invertible_RBG}
\end{figure}

\begin{rem}\label{rem:mirrors_examples}
    With Snappy we can check that 
    \begin{itemize}
        \item $K_{\pm1}=\pm K13n469=\pm v2272$,
        \item $K_{\pm2}=\pm t11462$, and
        \item $K_{\pm 3}= \pm o9\_41058$.
    \end{itemize}
    Thus each of these knots is a $4$-dimensional friend with its mirror.
\end{rem}

Next, we analyze which low-crossing and census knots are friends with their mirrors.

\begin{prop}\label{prop:mirror}
Let $K$ be a prime knot with at most $15$ crossings or a hyperbolic knot with tetrahedral complexity at most $9$. Then $K$ and $-K$ have orientation-preserving diffeomorphic $0$-surgeries if and only if $K$ is amphicheiral or if $K$ is isotopic to one of the following six chiral knots: $K13n469=v2272$, $K14n3411$, $K15n64176$, $t11462$,  $o9\_22951$, and $o9\_41058$.
\end{prop}

\begin{proof}
If $K$ and $K'$ have orientation-preserving diffeomorphic $0$-surgeries then their signatures agree $\sigma(K)=\sigma(K')$. For any knot $K$ we have $\sigma(K)=-\sigma(-K)$. Thus if $K$ and $-K$ have orientation-preserving diffeomorphic $0$-surgeries, then the signature of $K$ has to vanish.

We compute the signatures of all prime knots with at most $15$ crossings. Whenever the signature is non-vanishing then $K$ and $-K$ are not friends. The signature is vanishing for exactly $80015$ knots. We use SnapPy to check that $353$ of the remaining knots are actually isotopic to their mirrors. Among the remaining $79662$ knots exactly $44$ have non-hyperbolic $0$-surgeries. If any of the latter admits an orientation-reversing diffeomorphism, this diffeomorphism is isotopic to one that preserves the JSJ decomposition. In particular, there exist orientation-reversing diffeomorphisms of all JSJ pieces. If a piece of the JSJ decomposition is a Seifert fibered space, then we can use the classification of Seifert fibered spaces that admit orientation-reversing diffeomorphisms, see for example Section 8 in~\cite{Neumann_Raymond_SFS}. Thus we can read off from the Regina names in Table~\ref{tab:non_hyp_surgeries} directly if a Seifert fibered JSJ piece has an orientation-reversing diffeomorphism. If a JSJ decomposition admits a hyperbolic piece then we can use the symmetry group calculations in SnapPy of that piece to check if it admits an orientation-reversing diffeomorphism. By sorting out the $0$-surgeries that have a JSJ decomposition with at least one Seifert fibered or hyperbolic piece without orientation-reversing diffeomorphism, we are left only with the knot $K13n469$ whose $0$-surgery has the JSJ decomposition consisting of the figure eight knot complement $m004$ and the Seifert fibered space $SFS [D: (2,1) (2,-1)]$. Both spaces admit orientation-reversing diffeomorphisms. But since Dunfield's Regina code does not return the gluing map, we do not directly know if that space admits an orientation-reversing diffeomorphism. But by Remark~\ref{rem:mirrors_examples}, we know that $K13n469$ is actually the first knot in a family of infinitely many knots that are friends with their mirrors.

It remains to handle the chiral knots with vanishing signature and hyperbolic $0$-surgeries. For these, we use SnapPy to compute the symmetry groups of the $0$-surgeries from which we can read off if the $0$-surgeries admit an orientation-reversing diffeomorphism. It turns out that $K14n3411$ and $K15n64176$ are the only chiral knots that have amphicheiral $0$-surgeries.

The same strategy works also for the census knots. Of the $1267$ census knots exactly $139$ have vanishing signatures. Among these, $5$ are isotopic to their mirrors. Of the remaining knots, $56$ have non-hyperbolic $0$-surgery and $78$ have hyperbolic $0$-surgery. Of the non-hyperbolic $0$-surgeries we can exclude as above all but two ($v2272$ and $o9\_22951$) to admit an orientation-reversing diffeomorphism. The knot $v2272$ is in fact $K13n469$ and thus is a friend of its mirror by Remark~\ref{rem:mirrors_examples}. The $0$-surgery of $o9\_22951$ has JSJ decomposition consisting of the figure eight knot complement $m004$ and the Seifert fibered space $SFS [D: (3,1) (3,-1)]$. In Theorem~\ref{thm:Piccirillo_duals} below we show that $o9\_22951$ is a friend of its mirror by constructing an RBG link for that pair.

For the knots with hyperbolic $0$-surgeries, we use SnapPy to compute their symmetry groups and deduce that only two admit an orientation-reversing diffeomorphism. The knots $t11462$ and $o9\_22951$ both have hyperbolic $0$-surgeries that admit orientation-reversing diffeomorphisms. Thus they are friends with their mirrors. 

In total, we have found $6$ chiral knots among the low-crossing and census knots that are friends with their mirrors. (To verify that these knots are really not isotopic to their mirrors we compute their Jones polynomials.)

We also remark that the knots $K14n3411$, $K15n64176$, $t11462$, and $o9\_22951$ are chiral but have amphicheiral $0$-surgeries. It follows that these knots necessarily have $0$ as a symmetry-exceptional slope~\cite{BKMb}, cf.\ Proposition~\ref{prop:ex_sym}, since the symmetry groups of their $0$-surgeries are larger than the symmetry groups of the knots. Indeed, we have $\operatorname{Sym}(t11462)=\Z_2$ and $\operatorname{Sym}(t11462(0))=D_4$. For the other three knots, the symmetry groups are trivial, while their $0$-surgeries have symmetry group $\Z_2$.
\end{proof}

\section{Burton's and Thistlethwaite's lists}

To prove Theorem~\ref{thm:cc} and for the creation of the census from Theorem~\ref{thm:census} we need to show that some low crossing knots do not have friends of larger crossing numbers. More precisely, let $K$ be a prime knot with crossing number $c(K)$ at most $9$, we need to show that there exists no prime knot $K'$ with crossing number $c(K')\leq 25-c(K)$ such that $K$ and $K'$ have diffeomorphic $0$-surgeries. 

For that we first remark that the unknot~\cite{KMOS_char_unknot}, the trefoil $K3a1$~\cite{Ozvath_Szabo_trefoil_figeight}, the figure-eight knot $K4a1$~\cite{Ozvath_Szabo_trefoil_figeight}, and $K5a1$~\cite{Baldwin_Sivek_52} do not have any friends. So the first knot, that might have a friend is the $(5,2)$-torus knot $K5a2$. (However, it is conjectured that it has no friend at all.)

So to finish the proof of Theorem~\ref{thm:cc} we need to check that $K5a2$ has no friend among the prime knots with at most $20$ crossings, that no prime knot with $6$ crossings has a friend among the prime knots with at most $19$ crossings, and so on.

For that we run through Burton's lists of prime knots of at most $19$ crossings~\cite{Burton} and Thistlethwaite's list of prime knots with $20$ crossings~\cite{Thistlethwaite}, compute their Alexander polynomials and if one of the Alexander polynomials turns out to be the same as the Alexander polynomial of a knot with low-crossing number, we distinguish the $0$-surgeries by using the hyperbolicity status, the volume, the Regina names, or the fundamental groups as in Section~\ref{sec:low}. If these invariants all agree, we use SnapPy and Regina to search for a diffeomorphism between their $0$-surgeries. This successfully determines for all possible pairs of knots if they are friends or not.

\section{Census knots}
The same strategy as for the low-crossing knots works also for the census knots. But here everything works much faster since there are only $1267$ census knots whose complements can be triangulated by at most $9$ tetrahedra. Moreover, the classification of the exceptional fillings on the census knots was already done by Dunfield~\cite{Dunfield_census}, and many of its other invariants were already computed in~\cite{ABG+19,Baker_Kegel_braid_positive,CensusKnotInfo}. 
With the same approach, we also classify all census knots that have a friend among the knots with at most $15$ crossings.

This finishes the creation of our census from Theorem~\ref{thm:census} and thus also the proofs of Theorems~\ref{thm:cc} and~\ref{thm:tt}. 

\section{Smooth 4-genus}

It remains to perform the $4$-genus calculations in Theorem~\ref{thm:census}. For the knots with at most $12$ crossings and for the census knots we can read off the data from existing databases~\cite{KnotInfo,ABG+19,Baker_Kegel_braid_positive,CensusKnotInfo}. For the other knots, we search for a ribbon description using the code from~\cite{Ribbon_ML}. If this code could not identify a knot as being ribbon, we checked the standard sliceness obstructions, such as the signature, the Fox--Milnor condition, the $\tau$- and $s$-invariant, and the obstruction from~\cite{HKL_slice_obstruction}. This was enough to determine the sliceness status of all but the two knots ($o9\_22951$ and $o9\_41058$) in Table~\ref{tab:doubles}. The two knots with undetermined sliceness status are friends of their mirrors and so they have the same $4$-genus.

To show that all non-slice knots and the two knots with undetermined sliceness status in that table have smooth $4$-genus at most $1$, we search for a genus-$1$ concordance to a ribbon knot. This is done by searching for a crossing change that yields a knot that is identified by~\cite{Ribbon_ML} to be ribbon.

\section{Extending diffeomorphisms over traces and symmetry-exceptional slopes}

In the remaining part of this article, we ask which of the friends from Table~\ref{tab:doubles} are $4$-dimensional friends, i.e.\ share the same trace. In this section, we explain general strategies to study the question if two friends are $4$-dimensional friends.

To show in practice that two friends are $4$-dimensional friends, we describe a $0$-surgery diffeomorphism $\varphi\colon K(0)\rightarrow K'(0)$ via explicit Kirby moves and then analyze the image of the meridian $\mu_K$ of $K$ under $\varphi$. This can reveal that $\varphi$ extends to a diffeomorphism of the $0$-traces, see for example~\cite{Akbulut,AJOT_annulus_twisting_2013,Manolescu_Piccirillo_0_surgery}.

It is often harder to show that two friends are not $4$-dimensional friends. We first recall a result of Boyer~\cite{Boyer}.
Let $K$ and $K'$ be friends and let $\varphi\colon K(0)\rightarrow K'(0)$ be a diffeomorphism. We define the \textbf{parity} of $\varphi$ to be the parity of the intersection form of the closed $4$-manifold
\begin{equation*}
    X(K') \cup_\varphi -X(K).
\end{equation*}

 \begin{thm} [Boyer~\cite{Boyer}, cf.~Theorem 3.7 in~\cite{Manolescu_Piccirillo_0_surgery}] \label{thm:Boyer}
     $\varphi$ extends to a homeomorphism $\Phi\colon X(K)\rightarrow X(K')$ of the traces if and only if $\varphi$ is even.
 \end{thm}

We remark that there could exist $0$-surgery diffeomorphisms that extend to trace homeomorphisms but not to trace diffeomorphisms. Moreover, there exist examples of friends that have two different $0$-surgery homeomorphisms $\varphi,\varphi'\colon K(0)\rightarrow K'(0)$, such that $\varphi$ is odd and thus does not extend to a homeomorphism of the traces but such that $\varphi'$ is even and even extends to a trace diffeomorphism. We refer to Examples~\ref{ex:MP_ex_not4-d} and~\ref{ex:second_ex} for concrete examples.

On the other hand,~\cite {Manolescu_Piccirillo_0_surgery} describe situations in which every $0$-surgery diffeomorphism extends to a trace homeomorphism.

\begin{thm}[Manolescu--Piccirillo~\cite{Manolescu_Piccirillo_0_surgery}]\label{thm:Arf}
    If the Arf invariant $\operatorname{Arf}(K)$ of $K$ is non-vanishing, then any $0$-surgery diffeomorphism from $K(0)$ to any other $0$-surgery $K'(0)$ extends to a trace homeomorphism.
\end{thm}

Note that the Arf invariant of a knot can be computed from its Alexander polynomial and since the Alexander polynomial of a knot is an invariant of its $0$-surgery, it follows that this condition is symmetric. We have used SnapPy to compute the Arf invariants of all pairs of friends from our census. The friends with non-vanishing Arf invariants are displayed in Table~\ref{tab:Arf}. By Theorem~\ref{thm:Arf} these friends have homeomorphic traces.

\begin{table}[htbp] 
	\caption{The friends with non-trivial Arf invariant}
	\label{tab:Arf}
\begin{tabular}{|ccc|}
\hline
$(K6a1, 19nh\_78)$ & $(K14n5084, o9\_37547)$ & $(t11900, o9\_40803)$\\
$(K11n38, v3093)$ & $(K15n94464, o9\_40081)$ & $(o9\_22951, o9\_22951)$\\
$(K13n2527, K15n9379)$ & $(v2869, t12388)$&\\
\hline
\end{tabular}
\end{table}

On the other hand, to show that $K$ and $K'$ have non-homeomorphic traces we need to check that \textit{all} $0$-surgery diffeomorphisms $\varphi\colon K(0)\rightarrow K'(0)$ are odd. 
In the following, we present a criterium to show that any $0$-surgery diffeomorphism of two friends is odd and thus the friends are not $4$-dimensional friends. 

First, we recall that if $K$ is a hyperbolic knot, then Thurston's hyperbolic Dehn filling theorem tells us that for all but finitely many slopes $r$, the Dehn filling $K(r)$ of $K$ with slope $r$ is again hyperbolic~\cite{Thurston}. (The finitely many slopes for which $K(r)$ is not hyperbolic are called \textit{exceptional}.) Moreover, if the slope $r$ is sufficiently long then the core of the newly glued-in solid torus is the shortest geodesic of $K(r)$. By Mostow rigidity, any diffeomorphism of a hyperbolic manifold is isotopic to a unique isometry. Since an isometry preserves the shortest geodesic, it follows that for sufficiently large $r$, any symmetry of $K(r)$ restricts to a symmetry of $K$. If there exists a symmetry of $K(r)$ that does not restrict to a symmetry of $K$ then $r$ is called \textit{symmetry-exceptional slope}. In particular, it follows that if $K(0)$ has trivial symmetry group, then $0$ is not a symmetry-exceptional slope. We refer to~\cite{BKMb} for further discussion. We also mention that there exist computational methods to check if a given slope is a symmetry-exceptional slope~\cite{BKMb}.

\begin{prop}\label{prop:ex_sym}
Let $K$ and $K'$ be friends. 
    If $K$ is hyperbolic and $0$ is not a symmetry-exceptional slope of $K$, then all $0$-surgery diffeomorphisms $K(0)\rightarrow K'(0)$ have the same parity.
\end{prop}

\begin{rem}
    Opposed to Theorem~\ref{thm:Arf}, the condition in Proposition~\ref{prop:ex_sym} is not symmetric. There exist friends $K$ and $K'$ such that $0$ is a symmetry-exceptional slope for $K$ but not for $K'$, see for example Table~\ref{tab:symmetry_groups}.

    In Example~\ref{ex:MP_ex_not4-d} below we use Proposition~\ref{prop:ex_sym} to demonstrate that certain friends are not $4$-dimensional friends. 
\end{rem}

\begin{proof}[Proof of Proposition~\ref{prop:ex_sym}]
    Let $\varphi,\varphi'\colon K(0)\rightarrow K'(0)$ be two $0$-surgery diffeomorphisms. We write $f\colon K(0)\rightarrow K(0)$ for the unique diffeomorphism such that $\varphi'=\varphi\circ f$. By Mostow rigidity, $f$ is isotopic to a unique isometry and thus we can see $f$ as an element in the symmetry group of $K(0)$. Since we assume that $0$ is not a symmetry-exceptional slope of $K$, we conclude that $f$ restricts to an isometry of the complement $S^3\setminus K$ of $K$. By the solution of the knot complement problem~\cite{Gordon_Luecke}, such a symmetry sends a $0$-framed meridian $\mu_K$ of $K$ to itself. We conclude that $\varphi(\mu_K)$ is isotopic as framed knot to $\varphi'(\mu_K)$. Now we observe, that we obtain the closed $4$-manifold $Z_\varphi=X(K') \cup_\varphi -X(K)$ by attaching $2$-handles to $D^4$ along the $0$-framed $K'$ and the image of the $0$-framed meridian $\mu_K$ of $K$ under $\varphi$ and capping of with a $4$-handle. But since $\varphi(\mu_K)$ is isotopic to $\varphi'(\mu_K)$, the closed $4$-manifolds $Z_\varphi$ and $Z_{\varphi'}$ are diffeomorphic and thus $\varphi$ and $\varphi'$ have the same parity by Theorem~\ref{thm:Boyer}.
\end{proof}

For applying the obstruction from Proposition~\ref{prop:ex_sym} we classify the symmetry-excep\-tion\-al $0$-surgeries in our census.

\begin{prop}
    The classification of exceptional and symmetry-exceptional $0$-surgeries and their symmetry groups of the friends from Table~\ref{tab:doubles} is as displayed in Table~\ref{tab:symmetry_groups}.
\end{prop}

\begin{proof}
    We use the functions in SnapPy to compute the symmetry groups of the knots in our census and to verify the hyperbolicity of their $0$-surgeries. If the latter fails, we use Dunfield's Regina recognition code~\cite{Dunfield_census,FPS_code} to recognize the $0$-surgery as a non-hyperbolic manifold confirming that it is exceptional. For the hyperbolic $0$-surgeries we use the strategies explained in~\cite{BKMb} to compute their symmetry groups. This process is not always guaranteed to terminate, but for the knots in the census, it terminates quickly and yields verified results.
\end{proof}

\begin{table}[htbp]
	\caption{This table displays the symmetry groups of the knots from Table~\ref{tab:doubles} and the types and symmetry groups of the $0$-surgeries.}
	\label{tab:symmetry_groups}
 {\scriptsize
\begin{tabular}{|c|c||c|c||c|c|}
\hline
$K_1$ & $\operatorname{Sym}(K_1)$   & $K_2$ & $\operatorname{Sym}(K_2)$ & $K_1(0)=K_2(0)$ & $\operatorname{Sym}(K_1(0))$  \\  
\hline
\hline
$K6a1$ & $D_4$ &$19nh\_78$ & $0$ & hyperbolic & $D_4$  \\  
\hline

$K9n4$ & $\Z_2$  & $o9\_34801$ & $0$ & hyperbolic & $\Z_2^2$ \\
\hline

$K10n10$ & $0$  & $-t11532$ & $0$ & hyperbolic & $0$  \\  
\hline
$K10n10$ & $0$  &$o9\_43446$ & $0$ & hyperbolic & $0$\\ 
\hline
$K10n13$ & $\Z_2$ &  $o9\_34818$ & $0$ & $SFS \,[S^2: (2,1) (5,1) (10,-7)]$  & at least $\Z_2$ \\  	
\hline

$K11n38$ & $\Z_2$  &$v3093$ & $\Z_2$ & hyperbolic & $\Z_2^2$ \\  			
\hline
$K11n49$ & $\Z_2$ & $K15n103488$ & $\Z_2$ & hyperbolic & $\Z_2^2$ \\  
\hline
$K11n49$ & $\Z_2$  & $o9\_42735$ & $\Z_2$ & hyperbolic &  $\Z_2^2$\\
\hline
$K11n116$ & $0$  & $-t12607$ & $0$ & hyperbolic & $\Z_2$\\ 		
\hline	

$K12n121$ & $\Z_2$ & $-t11441$ & $\Z_2$  & $SFS \, [D^2: (2,1) (2,1)] \cup_m $ & at least $\Z_2$ \\ 

&&&& $SFS\, [D^2: (2,1) (3,1)], \,m = [ 5,1 | 4,1 ]$ &\\
\hline
$K12n200$ & $\Z_2$  &$t11748$ & $\Z_2$ & hyperbolic & $\Z_2^2$\\  					
\hline
$K12n309$ & $\Z_2$  &$-K14n14254$ & $\Z_2$ & hyperbolic &  $\Z_2$\\  					
\hline
$K12n318$ & $0$ &  $-o9\_39433$ & $0$ & hyperbolic & $0$ \\  	
\hline

$K13n469$ & $\Z_2$  &$-K13n469$ & $\Z_2$ & $JSJ\big(SFS \,[D^2: (2,1) (2,-1)]\cup m004\big)$ & at least $\Z_2$ \\  					
\hline
$K13n572$ & $\Z_2$ & $-K15n89587$ & $0$ & hyperbolic & $\Z_2^2$\\  					
\hline
$K13n1021$ & $\Z_2$ &  $K15n101402$ & $0$ & hyperbolic &$\Z_2^2$  \\
\hline
$K13n1021$ & $\Z_2$  & $-o9\_34908$ & $\Z_2$ & hyperbolic &$\Z_2^2$ \\ 
\hline
$K13n2527$ & $0$  &$-K15n9379$ & $\Z_2$ & hyperbolic & $\Z_2^2$\\  					
\hline
$K13n3158$ & $0$ & $o9\_41909$ & $0$ & hyperbolic & $\Z_2$ \\  	
\hline

$K14n3155$ & $\Z_2$  &$-K14n3155$ & $\Z_2$ & hyperbolic & $D_4$ \\  				
\hline
$K14n3155$ & $\Z_2$  &$t11462$ & $\Z_2$ & hyperbolic & $D_4$\\  				
\hline
$K14n3611$ & $\Z_2$  & $-o9\_27542$  &$\Z_2$ & $JSJ\big(M \Tilde\times S^1\cup  m015 \big)$ & at least $\Z_2$ \\ 				
\hline
$K14n5084$ & $\Z_2$  & $-o9\_37547$ & $\Z_2$ & hyperbolic & $\Z_2^2$\\ 				
\hline	

$K15n19499$ & $\Z_2$  &$K15n153789$& $0$ & $JSJ\big(SFS \,[D^2: (2,1) (2,-1)]\cup  m032\big)$ & at least $\Z_2$\\ 
\hline
$K15n19499$ & $\Z_2$  & $o9\_31828$ & $\Z_2$ & $JSJ\big(SFS \,[D^2: (2,1) (2,-1)]\cup  m032\big)$ & at least $\Z_2$ \\ 
\hline
$K15n94464$ & $0$  &$-o9\_40081$ & $\Z_2$ & hyperbolic & $\Z_2^2$ \\  	
\hline
$K15n101402$ & $0$  & $-o9\_34908$ & $\Z_2$ & hyperbolic & $\Z_2^2$ \\ 
\hline
$K15n103488$ & $\Z_2$ & $o9\_42735$ & $\Z_2$ & hyperbolic &  $\Z_2^2$\\  
\hline
$K15n153789$& $0$  &$o9\_31828$ & $\Z_2$ & $JSJ\big(SFS \,[D^2: (2,1) (2,-1)]\cup  m032\big)$ & at least $\Z_2$\\ 
\hline

$v0595$ & $\Z_2$ & $t12120$ & $\Z_2$ & $SFS \,[S^2: (2,1) (7,2) (14,-11)]$ & at least $\Z_2$ \\ 	
\hline
$v2869$ & $\Z_2$  & $t12388$ & $0$ & hyperbolic & $\Z_2$\\  	
\hline

$t07281$ & $\Z_2$  & $o9\_34949$ & $\Z_2$ & hyperbolic & $\Z_2^2$ \\  
\hline
$t07281$ & $\Z_2$  & $o9\_39806$ & $\Z_2$ & hyperbolic & $\Z_2^2$ \\
\hline
$-t10974$ & $\Z_2$ & $o9\_39967$ & $\Z_2$ & hyperbolic & $\Z_2$ \\  				
\hline	
$t11462$ &  $\Z_2$ & $-t11462$ & $\Z_2$ & hyperbolic &$D_4$\\  				
\hline
$t11532$ & $0$  & $-o9\_43446$ & $0$ & hyperbolic & $0$ \\ 
\hline
$t11900$ &  $0$ &$-o9\_40803$ & $0$ & hyperbolic & $0$\\  					
\hline

$o9\_22951$ & $\Z_2$  &$-o9\_22951$ & $\Z_2$ & $JSJ\big(SFS \,[D^2: (3,1) (3,-1)]\cup  m004\big)$ & at least $\Z_2$ \\  					
\hline
$o9\_39806$  & $\Z_2$ & $o9\_34949$ & $\Z_2$ &  hyperbolic&$\Z_2^2$ \\ 
\hline
$o9\_41058$ & $0$ &  $-o9\_41058$ & $0$ & hyperbolic& $\Z_2$\\  					
\hline
\end{tabular}
}
\end{table}

\section{Algorithmic search for 0-surgery diffeomorphisms}

In this section, we explain and implement various methods to algorithmically search for a Kirby calculus presentation of a $0$-surgery diffeomorphism of a given pair of friends. Using these methods we find explicit $0$-surgery diffeomorphisms for many friends from the census of Table~\ref{tab:doubles} and are in some cases able to answer if the friends are $4$-dimensional friends or not.

We also note here that there exists an algorithm that takes as input diagrams of a pair of friends $K$ and $K'$ and outputs a $0$-surgery diffeomorphism presented as a sequence of Kirby moves. Indeed, we can just enumerate all possible Kirby moves in the Kirby diagram of $K(0)$ given by a $0$-framed diagram of $K$. In finite time, we will reach the Kirby diagram of $K'(0)$ given by a $0$-framed diagram of $K'$. Of course, this algorithm is not practical but guaranteed to theoretically always work. Below we explain certain strategies that will not work in general but often return a verified result in reasonable time.

\subsection{Piccirillo friends}

Let $K$ be a knot with unknotting number $1$. Piccirillo \cite{Piccirillo_Conway_knot} describes a simple method for constructing a knot $K^*$, the \textit{Piccirillo friend}, whose trace is diffeomorphic to the trace of $K$. Given a diagram $D$ of $K$ that realizes the unknotting number of $K$ this process is algorithmic. We have implemented that algorithm~\cite{GitHub} and have applied that code to prove the following. 

\begin{thm}\label{thm:Piccirillo_duals}
At least one of the knots of the $9$ pairs of friends displayed in Table~\ref{tab:Piccirillo} has unknotting number one and the other knot is its Piccirillo friend. Thus these friends have diffeomorphic traces.
\end{thm}

\begin{table}[htbp] 
	\caption{The Piccirillo friends}
	\label{tab:Piccirillo}
\begin{tabular}{|ccc|}
\hline
$(K6a1,19nh\_78)$ & $(K13n469,-K13n469)$ & $(K15n153789,o9\_31828)$\\
$(K9n4,o9\_34801)$ & $(K13n2527,-K15n9379)$ & $(t07281,o9\_34949)$\\
$(K11n38,v3093)$ & $(K15n94464,-o9\_40081)$&$(o9\_22951,-o9\_22951)$\\
\hline
\end{tabular}
\end{table}

\begin{proof}
Of the knots in Table~\ref{tab:doubles} we have found that at least $15$ knots have unknotting number one. For that, we load (potential) minimal crossing diagrams of these knots from~\cite{Burton,Baker_Kegel_braid_positive} and check which of these diagrams have unknotting crossings. We found that exactly $15$ of these diagrams admit an unknotting crossing. The other diagrams are proven to have no unknotting crossing and we expect the knots to have unknotting number larger than one. 

Next, we implement the algorithm from~\cite{Piccirillo_Conway_knot} and construct the Piccirillo friend $K^*$. Then we use SnapPy to check if the Piccirillo friend is isotopic to the friend in Table~\ref{tab:doubles}. This identifies that the above-claimed $9$ friends are $4$-dimensional friends. Note that this also yields a simple (human checkable) certificate that these knots share the same $0$-surgeries. 

In the other cases, the Piccirillo friend was a different knot than the one appearing in our census.
\end{proof}

\subsection{Special RBG links}

Manolescu--Piccirillo~\cite{Manolescu_Piccirillo_0_surgery} describe a particularly simple family of RBG links, called \textit{special} RBG links, with the property that $B$ and $G$ both are meridians of $R$. We transform a diagram of a special RBG link $L$ into a Kirby diagram by putting an integer framing $r$ on $R$ and framing $0$ on both $G$ and $B$. Then we can perform a slam dunk to cancel $R$ and $B$. The image of $G$ under this cancellation is a knot $K_G$ in $S^3$. By symmetry, the same holds true with the roles of $G$ and $B$ reversed. By choosing the correct framing $r$ on $R$ we get two knots $K_B$ and $K_G$ that share the same $0$-surgery.

Now the family of RBG links $L_1{(a,b,c,d,e,f)}$ from~\cite{Manolescu_Piccirillo_0_surgery} shown in Figure~\ref{fig:special_RBG}(i) depends on six parameters $(a,b,c,d,e,f)$ and is chosen in such a way that the resulting knots $K_B$ and $K_G$ have relatively small crossing number. In Figure~\ref{fig:special_RBG}(ii) and (iii) we display two more similar RBG links $L_2{(a,b,c,d,e,f)}$ and $L_3{(a,b,c,d,e,f)}$. It is shown in~\cite{Manolescu_Piccirillo_0_surgery} that the parity of $r=a+b$ yields the parity of the associated $0$-surgery diffeomorphism (given by the slam dunks) and that if $r=0$ this $0$-surgery diffeomorphism extends to a trace diffeomorphism. Below we use this family to find explicit Kirby moves for several pairs of friends from our census and study if they are $4$-dimensional friends or not.

\begin{figure}[htbp] 
	\centering
	\def\svgwidth{0,99\columnwidth}
	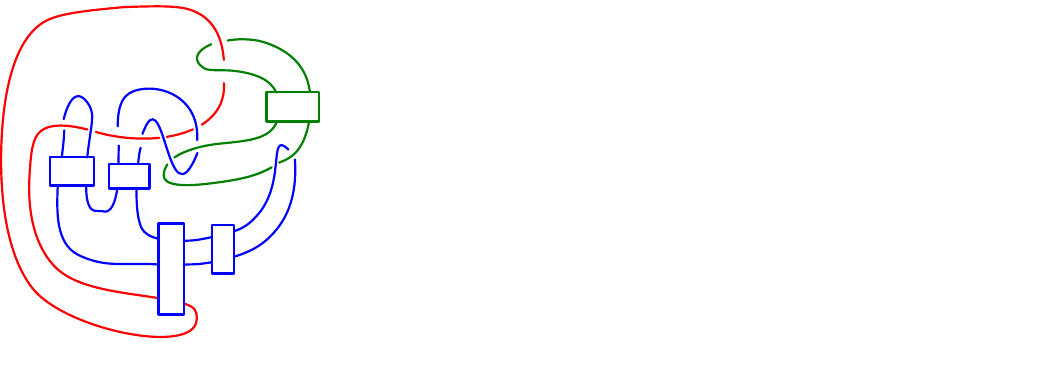
	\caption{Three special RBG links. $L_1$ is the RBG link from Figure 13 in~\cite{Manolescu_Piccirillo_0_surgery}. $L_2$ and $L_3$ are obtained from $L_1$ by crossing changes.}
	\label{fig:special_RBG}
\end{figure}

\begin{thm}\label{thm:specialRBG}
    The $16$ pairs of friends shown in Table~\ref{tab:MP_data} arise as knots $K_B$ and $K_G$ from the RBG link $L_i{(a,b,c,d,e,f)}$ for the printed values of ${(a,b,c,d,e,f)}$.
    The column 'traces' describes how the traces of these friends are related.
\end{thm}

\begin{table}[htbp] 
	\caption{Friends obtained from special RBG links}
	\label{tab:MP_data}
 {\scriptsize
\begin{tabular}{|c|c|c|c|}
\hline
knot & knot &  $L_i(a,b,c,d,e,f)$ & traces  \\  
\hline
\hline
$K9n4$   & $o9\_34801$ &  
$L_2(0, 0, 1, 0, 1, 0)$; 
$L_1(-2, 1, -1, 1, 1, 0)$& $C^\infty$ \\
\hline

$K10n13$  & $o9\_34818$ & 
$L_1(2, -2, 0, 1, 1, 0)$ & $C^\infty$ \\
\hline

$K11n38$   & $v3093$ & $L_2(0, 0, 0, 0, 1, 0)$& $C^\infty$  \\
\hline
$K11n49$ &$K15n103488$ & 
$L_1(-2, 1, -1, 1, 0, 0)$&   \\
\hline
$K11n116$  & $-t12607$& 
$L_1(-1, 1, -1, 1, 0, 0)$;
$L_1(-2, 1, 0, 1, 0, 0)$ & $C^\infty$ \\
\hline	

$K12n121$  & $-t11441$ & 
$L_1(1, -2, 0, 1, 1, 0)$&  \\	
\hline
$K12n200$  & $t11748$ & 
$L_1(3, -2, -1, 1, 1, 0)$&  \\	
\hline

$K13n572$   & $-K15n89587$  &
$L_1(-1, 1, -3, 1, 0, 0 )$; 
$L_1(-2, 1, -1, 1, -1, 0)$ & $C^\infty$ \\
\hline
$K13n1021$   & $K15n101402$ &   
$L_1(-2, 1, -2, 1, -1, 0 )$ &  \\
\hline
$K13n1021$  & $-o9\_34908$ & $L_2(-2, 1, 0, 1, 1, -1)$  &  \\
\hline
$K13n2527$  &$-K15n9379$  & 
$L_1(-1, 1, -2, 1, 0, 0)$; $L_3(-2, 1, -1, 0, -1, 0)$ & $C^\infty$ \\ 
\hline
$K13n3158$  & $o9\_41909$ & 
$L_1(0, 1, -1, 1, 0, 0 )$ &  \\
\hline

$K14n3611$  & $-o9\_27542$& 
$L_1(2, -2, -1, 1, 1, 0 )$; 
$L_1(3, -2, -1, 1, 0, 0 )$; &$C^\infty$\\
\hline
$K14n5084$  & $-o9\_37547$ & $L_1(1, 1, 0, 0, 2, -1)$  & $C^0$ \\	
\hline	

$K15n19499$  &$K15n153789$&  
$L_1( -1, 1, -2, 1, -1, 0)$; $L_1(-5, 1, 0, 1, 1, 0)$ & $C^\infty$ \\
\hline
$K15n19499$  & $o9\_31828$& 
$L_1(-1, 1, -1, 1, 0, -1)$; 
$L_1(-2, 1, -1, 1, 1, -1)$ & $C^\infty$ \\
\hline

\end{tabular}
}
\end{table}

\begin{proof}
    We import a surgery presentation of $L_i{(a,b,c,d,e,f)}$ into SnapPy which we can Dehn fill to obtain ideal triangulations of the complements of the associated knots $K_B$ and $K_G$ for small values of the parameters $(a,b,c,d,e,f)$. Then we check if SnapPy recognizes these triangulations as complements of knots in the census or in the low-crossing knots. Concretely, we check each of the parameters $(a,b,c,d,e,f)$ in the range $\{-10,\ldots,10\}$. Most pairs of friends that arise from $L_i{(a,b,c,d,e,f)}$ arise in several different ways. From the RBG link, we get an explicit sequence of Kirby moves describing a $0$-surgery diffeomorphism. The parity of this diffeomorphism is the parity of $a+b$. In Table~\ref{tab:MP_data} we always list the simplest choice of parameters for each possible parity. If we found an even $0$-surgery diffeomorphism we always tried to list one that has $a+b=0$, since for these we know that the $0$-surgery diffeomorphism extends to a trace diffeomorphism.  
\end{proof}

\begin{rem}
    For the friends where we only found odd $0$-surgery diffeomorphisms, we can check from Table~\ref{tab:symmetry_groups} that the $0$-surgery is either exceptional or symmetry-exceptional and thus Proposition~\ref{prop:ex_sym} cannot be applied. A possible approach to study these cases would be to compute the symmetry group $G$ of the $0$-surgery. If $G$ is a finite group (for example if the $0$-surgery is hyperbolic) we consider all possible compositions of the $0$-surgery diffeomorphisms given by the special RBG links and see them as elements in $G$. Next, we follow the action of the meridians through these Kirby moves. From the action of the meridian under these compositions, we can distinguish the $0$-surgery diffeomorphisms and find a composition of $0$-surgery diffeomorphisms for every element in $G$. In particular, we can display also the exceptional symmetries like this. From the action on the meridians, we can also determine the parities of these $0$-surgery diffeomorphisms.
\end{rem}

\subsection{Flat annulus presentations}

A classical method to construct infinitely many mutually distinct knots that share the same $0$-surgery is via \textit{annulus twisting}~\cite{Osoinach_annulus, AJOT_annulus_twisting_2013,AJDLO_same_traces,AbeTagami_annulus_presentation}. For examples, we refer to the two links $L_1$ and $L_2$ shown in Figure~\ref{fig:flat_annulus}. Performing $(1/n)$- and $(-1/n)$-surgeries on the blue and red components yields a sequence of knots $K_n^m$. By the main result of~\cite{Osoinach_annulus} it follows that for every $m\in\Z$ the knots $K_n^m$ are all friends. Moreover, such an annulus presentation yields an explicit sequence of Kirby moves describing $0$-surgery diffeomorphisms $\varphi_n^m\colon K_0^m(0)\rightarrow K_n^m(0)$, see for example~\cite{Osoinach_annulus}. If the core curve of the annulus is unknotted and if the boundary curves of the annulus have vanishing linking number (as the red and blue curves in Figure~\ref{fig:flat_annulus}) it is straightforward to check that the parity of $\varphi_n^m$ is given by the parity of $n$ and that for any even $n$, the $0$-surgery diffeomorphism $\varphi_n^m\colon K_0^m(0)\rightarrow K_n^m(0)$ extends to a trace diffeomorphism, see for example~\cite{Manolescu_Piccirillo_0_surgery}. 

\begin{figure}[htbp] 
	\centering
	\def\svgwidth{0,99\columnwidth}
	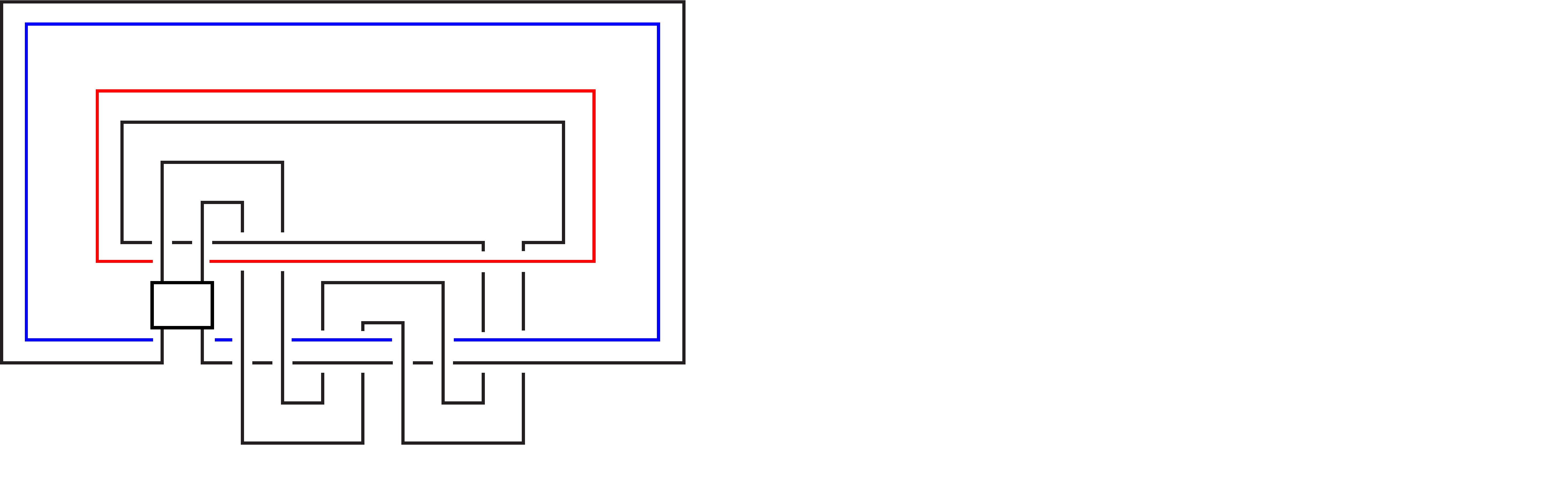
	\caption{Two flat annulus presentations $L_1$ and $L_2$.}
	\label{fig:flat_annulus}
\end{figure}

\begin{thm}\label{thm:annulus}
The pairs of friends shown in Table~\ref{tab:AT} arise as knots $K_n^m$ from the flat annulus presentations shown in Figure~\ref{fig:flat_annulus}.
The column ‘traces' describes how the traces of these friends are related.
\end{thm}

\begin{proof}
    We import the surgery descriptions $L_1$ and $L_2$ into SnapPy and use these to create ideal triangulations of the complements of $K_n^m$ for small values of $m$ and $n$. (Concretely, we have checked $n,m\in\{-10,10\}$.) Whenever SnapPy recognizes a pair of these manifolds (with the same $m$ but different values of $n$) as manifolds from our census, we print it in Table~\ref{tab:AT}. By the above discussion, the parity of the $0$-surgery diffeomorphism given by the Kirby moves induced from the annulus presentation is given by the parity of the difference of the values of $n$. If that difference is even, the traces are diffeomorphic and if it is odd, it follows that this particular $0$-surgery diffeomorphism does not extend to a trace homeomorphism. In Examples~\ref{ex:MP_ex_not4-d} and~\ref{ex:second_ex} below, we use Proposition~\ref{prop:ex_sym} to deduce that for some of these friends, no $0$-surgery diffeomorphism extends to a trace homeomorphism.
\end{proof}

\begin{table}[htbp] 
	\caption{Flat annulus twisting}
	\label{tab:AT}
 {\scriptsize
\begin{tabular}{|c|c||c|c||c|}
\hline
knot &  $L_i(m,n)$ & knot &  $L_i(m,n)$ & traces \\  
\hline
\hline
$K10n10$ & $L_1(-3, 1)$; $L_1(0, 0)$ & $-t11532$ & $L_1(-3, 0)$; $L_1(0, 1)$ & not $C^0$ \\
\hline
$K10n10$ & $L_1(-3, 1)$; $L_1(0, 0)$ & $o9\_43446$ & $L_1(-3, -1)$; $L_1(0, 2)$ & $C^\infty$ \\
\hline
$K11n49$ & $L_2(-3, 1)$; $L_2(-1, 0)$ & $K15n103488$ & $L_2(-3, 0)$; $L_2(-1, 1)$&  \\
\hline
$K11n49$ & $L_2(-3, 1)$; $L_2(-1, 0)$ & $o9\_42735$ & $L_2(-3,-1)$; $L_2(-1, 2)$& $C^\infty$ \\
\hline
$K11n116$ & $L_2(-4, 1)$; $L_2(0, 0)$ & $t12607$ & $L_2(-4, 0)$; $L_2(0, 1)$&  \\
\hline
$K12n309$ & $L_1(-2, 1)$; $L_1(-1, 0)$ & $K14n14254$ & $L_1(-2, 0)$; $L_1(-1, 1)$ &  \\
\hline
$K12n318$ & $L_1(-4,1)$; $L_1(1, 0)$ & $o9\_39433$ & $L_1(-4, 0)$; $L_1(1, 1)$ & not $C^0$ \\
\hline
$K13n469$ & $L_2(-2, 0)$ & $-K13n469$ & $L_2(-2, 1)$&  \\
\hline
$K13n3158$ & $L_2(-5, 1)$; $L_2(1, 0)$ & $o9\_41909$ & $L_2(-5, 0)$; $L_2(1, 1)$&  \\
\hline
$K15n103488$ & $L_2(-3, 0)$; $L_2(-1, 1)$ & $o9\_42735$ & $L_2(-3, -1)$; $L_2(-1, 2)$&  \\
\hline
$t11532$ & $L_1(-3, 0)$; $L_1(0, 1)$ & $o9\_43446$ & $L_1(-3, -1)$; $L_1(0, 2)$ & not $C^0$ \\
\hline
\end{tabular}
}
\end{table}

\begin{ex} \label{ex:MP_ex_not4-d}
From Table~\ref{tab:symmetry_groups} we read off that the three pairs of friends 
$$(K10n10,t11532),\, (K12n318,o9\_39433),\, \textrm{ and } \, (t11532,-o9\_43446)$$ are all asymmetric and have asymmetric $0$-surgery. Thus Proposition~\ref{prop:ex_sym} implies that all $0$-surgery diffeomorphisms have the same parity (since the $0$-surgeries are asymmetric, all $0$-surgery diffeomorphisms are even isotopic). From Table~\ref{tab:AT} we read off that these three pairs of friends each admit a $0$-surgery diffeomorphism with odd parity and thus their traces cannot be homeomorphic.

We consider the pair of friends $K12n309$ and $-K14n14254$. Both knots are strongly invertible and have symmetry group $\Z_2$ generated by the strong inversion. In Table~\ref{tab:symmetry_groups} we computed that the symmetry group of their $0$-surgery is $\Z_2$, generated by the extension of the strong inversion. In particular, $0$ is not symmetry-exceptional for both knots. By Proposition~\ref{prop:ex_sym} all $0$-surgery diffeomorphisms have the same pairity. Since the $0$-surgery diffeomorphism from Table~\ref{tab:AT} is odd, Boyer's theorem (Theorem~\ref{thm:Boyer}) implies that no $0$-surgery diffeomorphism of this pair extends to a homeomorphism of the traces. 
\end{ex}

\begin{ex}\label{ex:second_ex}
We consider the pair of friends $K13n469$ and $-K13n469$. The diffeomorphism of the $0$-surgery obtained from the annulus twist in Table~\ref{tab:AT} does not extend to a diffeomorphism of the traces (since it has odd parity). However, the diffeomorphism of the $0$-surgery from Remark~\ref{rem:mirrors_examples} extends to a diffeomorphism of the traces. Thus Proposition~\ref{prop:ex_sym} implies that $0$ is an exceptional or symmetry-exceptional slope of $K13n469$. In fact, we can check that the $0$-surgery of $K13n469$ is not hyperbolic, see Table~\ref{tab:symmetry_groups}. 
\end{ex}

Now the part about the traces in Theorem~\ref{thm:census} and thus also Theorem~\ref{thm:traces} follows by combining the results from Table~\ref{tab:Arf} and Theorems~\ref{thm:Piccirillo_duals}, \ref{thm:specialRBG} and~\ref{thm:annulus}.

\let\MRhref\undefined
\bibliographystyle{hamsalpha}
\bibliography{complex.bib}

\end{document}